%% file: kdv__arXiv.tex
\documentclass[]{scrartcl}

\usepackage[utf8]{luainputenc}
\usepackage[USenglish]{babel}
\usepackage{csquotes}

\usepackage[a4paper,top=27mm,bottom=20mm,inner=25mm,outer=20mm]{geometry}

\usepackage[%
  backend=bibtex,bibencoding=ascii,
  style=numeric-comp,
  giveninits=true, uniquename=init, 
  natbib=true,
  url=true,
  doi=true,
  isbn=false,
  backref=false,
  maxnames=99,
  ]{biblatex}
\usepackage{amsmath}
\allowdisplaybreaks
\numberwithin{equation}{section}
\usepackage{amssymb}
\usepackage{commath}
\usepackage{mathtools}
\usepackage{bbm}
\usepackage{nicefrac}
\usepackage{subdepth}

\usepackage{siunitx}
\usepackage{algorithm}
\usepackage{algpseudocode}

\usepackage{amsthm}

\usepackage[plainpages=false,pdfpagelabels,hidelinks,unicode]{hyperref}

\usepackage{cleveref}

\usepackage{thmtools}
\usepackage{etoolbox}
\makeatletter
\patchcmd{\thmt@setheadstyle}
 {\bgroup\thmt@space}
 {\thmt@space}
 {}{}
\patchcmd{\thmt@setheadstyle}
 {\egroup\fi}
 {\fi}
 {}{}
\makeatother
\declaretheoremstyle[
  bodyfont=\normalfont\itshape,
  headformat=\NAME\ \NUMBER\NOTE,
]{myplain}
\declaretheoremstyle[
  headformat=\NAME\ \NUMBER\NOTE,
]{mydefinition}
\newcommand{\envqed}{{\lower-0.3ex\hbox{$\triangleleft$}}}
\declaretheorem[style=myplain,numberwithin=section]{theorem}

\declaretheorem[style=mydefinition,numberlike=theorem,qed=\envqed]{remark}

\usepackage{color}
\usepackage{graphicx}
\usepackage[small]{caption}
\usepackage{subcaption}

\begingroup\expandafter\expandafter\expandafter\endgroup
\expandafter\ifx\csname pdfsuppresswarningpagegroup\endcsname\relax
\else
\fi

\usepackage{booktabs}
\usepackage{rotating}
\usepackage{multirow}

\usepackage{enumitem}

\usepackage{ifluatex}
\ifluatex
  \usepackage[no-math]{fontspec}
\else
  \usepackage[T1]{fontenc}
\fi
\usepackage{newpxtext,newpxmath}

\input{macros}

\newcommand{\orcid}[1]{ORCID:~\href{https://orcid.org/#1}{#1}}
\usepackage{authblk}

\newenvironment{keywords}{\par\textbf{Key words.}}{\par}
\newenvironment{AMS}{\par\textbf{AMS subject classification.}}{\par}

\title{Conserving mass, momentum, and energy for the Benjamin-Bona-Mahony, Korteweg-de Vries, and nonlinear Schr{\"o}dinger equations}

\author[1]{Hendrik~Ranocha\thanks{\orcid{0000-0002-3456-2277}}}
\affil[1]{Institute of Mathematics, Johannes Gutenberg University Mainz, Staudingerweg 9, 55128 Mainz, Germany}

\author[2]{David~I.~Ketcheson\thanks{\orcid{0000-0002-1212-126X}}}
\affil[2]{King Abdullah University of Science and Technology (KAUST),
Computer Electrical and Mathematical Science and Engineering Division (CEMSE),
Thuwal, 23955-6900, Saudi Arabia}

\date{April 10, 2026} 

\makeatletter
\hypersetup{pdfauthor={Hendrik Ranocha, David I. Ketcheson}} 
\hypersetup{pdftitle={Conserving mass, momentum, and energy for the BBM, KdV, and NLS equations}} 
\makeatother

\begin{document}

\maketitle

\begin{abstract}
\noindent
  \input{abstract.tex}
\end{abstract}

\begin{keywords}
  Fourier Galerkin methods,
  additive Runge-Kutta methods,
  structure-preserving methods,
  Benjamin-Bona-Mahony equation,
  Korteweg-de Vries equation,
  nonlinear Schr{\"o}dinger equation
\end{keywords}

\begin{AMS}
  65M60, 
  65M70, 
  65M12,  
  65M20 
\end{AMS}

\input{kdv.tex}

\appendix

\section*{Acknowledgments}

\input{funding}

We thank Mario Ricchiuto for hosting us in Bordeaux during the week
October 6--10 2025, where this project was initiated in breakfast
discussions.

\printbibliography

\appendix

\end{document}

%% file: macros.tex
\let\epsilon\varepsilon
\let\phi\varphi
\let\rho\varrho

\renewcommand{\i}{\mathrm{i}}

\providecommand\e{}
\renewcommand{\e}{\mathrm{e}}

\providecommand\R{}
\renewcommand{\R}{\mathbb{R}}

\newcommand{\I}{\operatorname{I}}

\renewcommand{\Im}{\operatorname{Im}}

\newcommand{\sech}{\operatorname{sech}}

\newcommand{\mass}{\mathcal{M}}
\newcommand{\momentum}{\mathcal{P}}
\newcommand{\energy}{\mathcal{E}}

%% file: abstract.tex
We propose and study a class of arbitrarily high-order numerical
discretizations that preserve multiple invariants and
are essentially explicit (they do not require the solution of any large systems
of algebraic equations).
In space, we use Fourier Galerkin methods, while in time we use a combination of orthogonal projection and relaxation.
We prove and numerically demonstrate the conservation properties of the method
by applying it to the Benjamin-Bona-Mahony, Korteweg-de Vries, and nonlinear Schr\"odinger (NLS) PDEs as well as a hyperbolic approximation of NLS.
For each of these equations, the proposed schemes conserve mass, momentum, and energy up to numerical precision.
We show that this conservation leads to reduced growth of numerical errors for long-term simulations.

%% file: kdv.tex
\section{Introduction}
Many important partial differential equations (PDEs) possess conserved
quantities (such as mass, momentum, or energy) that are fundamental properties
of the corresponding physical system.  Preserving these invariants by
numerical methods is essential not only in obtaining accurate solutions but
also to ensure that the solutions are physically meaningful at all.
Therefore, great effort has gone into the development of \emph{structure-preserving}
numerical methods.  Most often, such methods are designed to preserve one, or
in some cases, two such invariants.  However, many models possess more than two
invariants; indeed, fully integrable systems (such as the Korteweg-de Vries or
nonlinear Schr\"odinger equations) possess infinitely many.

In this work we present a class of full (space and time) discretizations
that conserve three invariants for three of the
most important nonlinear dispersive wave models:
the Benjamin-Bona-Mahony (BBM), Korteweg-de Vries (KdV),
and nonlinear Schr{\"o}dinger (NLS) equations.
These invariants are typically referred to as mass, momentum, and energy
in the literature, although it should be noted that they are in general
used by way of analogy and do not always
refer to physical mass, momentum and energy in applications
(see e.g. \cite{ablowitz1979evolution,karczewska2015energy,ali2014formulation}).
For periodic boundary conditions, each of these equations conserves
the total mass $\mass$, momentum $\momentum$, and energy $\energy$
given as follows.

The Benjamin-Bona-Mahony equation\footnote{This equation is often written with an additional linear term $+ u_x$, which can be removed by the transformation $u \mapsto u - 1$ to obtain the normalization also used in \cite{gavrilyuk2022hyperbolic,bleecke2025asymptotic}.} \cite{peregrine1966calculations,benjamin1972model}
\begin{equation}
\label{eq:bbm}
  u_t + u u_x - u_{txx} = 0
\end{equation}
has the three invariants \cite{olver1979euler}
\begin{equation}
\label{eq:bbm_invariants}
  \mass = \int u \dif x,
  \qquad
  \momentum = \int \left( \frac{1}{2} u^2 + \frac{1}{2} (u_x)^2 \right) \dif x,
  \qquad
  \energy = \int \frac{1}{6} u^3 \dif x.
\end{equation}
The Korteweg-de Vries equation \cite{boussinesq1872theorie,korteweg1895change}
\begin{equation}
\label{eq:kdv}
  u_t + u u_x + u_{xxx} = 0
\end{equation}
has a countably infinite number of invariants \cite{miura1968korteweg};
the first three of them are
\begin{equation}
\label{eq:kdv_invariants}
  \mass = \int u \dif x,
  \qquad
  \momentum = \int \frac{1}{2} u^2 \dif x,
  \qquad
  \energy = \int \left( \frac{1}{2} (u_x)^2 - \frac{1}{6} u^3 \right) \dif x.
\end{equation}
The nonlinear Schr{\"o}dinger equation \cite{sulem2007nonlinear,yang2010nonlinear}
\begin{equation}
\label{eq:nls-u}
  \i u_t + u_{xx} + \beta |u|^2 u = 0
\end{equation}
also has infinitely many invariants.
The first three of them are
\begin{equation}
\label{eq:nls-u_invariants}
  \mass = \int |u|^2 \dif x,
  \qquad
  \momentum = \int \Im (\overline{u} u_x) \dif x,
  \qquad
  \energy = \int \left( |u_x|^2 - \frac{\beta}{2} |u|^4 \right) \dif x.
\end{equation}

Our approach employs the widely-used method of lines, in which PDEs are
discretized in space and the resulting ordinary differential equation (ODE)
system is then integrated.
In order to preserve an invariant at the fully-discrete level within this
framework, both the spatial and the temporal discretization must be conservative.
Due to the spatial discretization, the quantity conserved by the numerical
method is some discrete approximation of the original invariant.

\subsection{Spatial discretizations}

Since conservation of invariants can be proven using integration by parts,
many conservative spatial discretizations are created by mimicking this
procedure at the discrete level, either using a Galerkin approach (assuming
exact integration of all nonlinear terms) or by using summation-by-parts (SBP)
operators \cite{fernandez2014review,svard2014review}.
While SBP-based methods can be constructed to conserve two invariants for the
equations of interest
\cite{ranocha2020relaxationHamiltonian,ranocha2021broad,linders2023resolving,ranocha2025high},
they do not appear to be able to conserve three or more invariants.
The underlying reason for this is that SBP discretizations conserving the
total mass are based on split forms of the nonlinear terms, which are related
to entropy-conserving methods for conservation laws in the classical setting
of Tadmor \cite{tadmor1987numerical,tadmor2003entropy}. The nonlinear term
of the BBM/KdV equation is the same as in Burgers' equation
$u_t + (u^2 / 2)_x = 0$. Using Tadmor's theory, we can construct numerical
fluxes that can either conserve the quadratic invariant $\int u^2 / 2$
or the cubic invariant $\int u^3 / 6$, but not both at the same time
(since entropy-conservative fluxes are determined uniquely for scalar
conservation laws).

One alternative approach to construct spatial discretizations conserving
multiple invariants is to use a reasonable baseline discretization and
add correction terms enforcing the desired conservation properties
\cite{abgrall2018general,abgrall2022reinterpretation}.
In this vein, Chen et al.\ \cite{chen2022new} introduced local discontinuous Galerkin (LDG) methods for
KdV with additional unknown stabilization parameters to enforce conservation of
the first three invariants.
However, initial numerical experiments in which we have extended this approach to
fully-discrete conservation suggest that this approach is less robust
in practice, at least when combined with the temporal discretizations described below.

The Ablowitz-Ladik lattice \cite[Ch.~3]{ablowitz2004discrete} can also be
viewed as a spatial discretization of NLS that possesses an infinite set of
conserved quantities related to those of NLS.

Here we turn instead to Galerkin methods. Many classical finite element schemes
use piecewise polynomials with a prescribed degree of regularity at cell boundaries;
since the derivative of such a polynomial does not lie in the same space
(it has lower regularity), such methods cannot be used to conserve invariants involving
derivatives.  However, Fourier Galerkin methods are promising
since the derivative is an endomorphism on the space of trigonometric
polynomials.  Indeed, Maday and Quarteroni \cite[Lemma~II.1]{maday1988error} showed that the
Fourier Galerkin method conserves the first three KdV invariants.
Here we extend this result to the NLS and BBM equations, and
provide efficient time discretizations that lead to a fully-discrete conservative scheme.
Our focus is in constructing fully-discrete methods conserving three invariants;
for further refined analyses of the convergence properties of Fourier semidiscretizations building upon the work of Maday and Quarteroni for the KdV equation, we refer to \cite{bjorkavaag2007exponential,kalisch2007rate}.

\subsection{Temporal discretizations}

For temporal conservation of invariants, the literature on structure-preserving
ODE integrators  is extensive; for an overview we refer the reader to the
monograph \cite{hairer2006geometric}.  We will highlight some of the most
relevant approaches.
Linear or affine invariants are automatically preserved by the most common types of discretizations, e.g., general linear methods such as Runge-Kutta methods and linear multistep methods.
Special implicit methods can be designed to conserve quadratic invariants (symplectic methods \cite[Sections~IV.2 and VI.7]{hairer2006geometric})
or to preserve the Hamiltonian in the case of Hamiltonian systems (e.g., discrete gradient methods (that can also preserve other functionals) or the average vector field method \cite{mclachlan1999geometric,quispel2008new,celledoni2009energy,hairer2010energy}).
We mention here also the scalar auxiliary variable (SAV) method (e.g.,
\cite{li2021linear}) in which the equations to be solved are augmented by one
or more additional equations related to the conserved quantity or quantities.
For more general invariants, one can simply project the
solution back onto the conservative manifold after each step.  This can be done using
orthogonal projection \cite[Section~IV.4]{hairer2006geometric} or by projecting along a line determined by the numerical
ODE solver; the latter approach is known as relaxation
\cite{ketcheson2019relaxation,ranocha2020relaxation,ranocha2020general}.

The basic idea of relaxation methods dates back to \cite{sanz1982explicit} and
\cite[pp.~265--266]{dekker1984stability}, and has recently been developed in
a general setting in \cite{ketcheson2019relaxation,ranocha2020relaxation,ranocha2020general}.
It has been combined with Runge-Kutta methods \cite{ranocha2020relaxation},
linear multistep methods \cite{ranocha2020general},
residual distribution schemes \cite{abgrall2022relaxation},
implicit-explicit (IMEX) methods \cite{kang2022entropy,li2022implicit},
and multi-derivative methods \cite{ranocha2023functional,ranocha2024multiderivative}.
Some applications include
Hamiltonian problems \cite{ranocha2020relaxationHamiltonian,zhang2020highly,li2023relaxation},
compressible flows \cite{yan2020entropy,ranocha2020fully,doehring2025paired},
dispersive wave equations \cite{li2025time,ranocha2025structure,lampert2025structure,mitsotakis2021conservative},
and asymptotic-preserving methods for hyperbolizations \cite{biswas2025traveling,bleecke2025asymptotic,giesselmann2026convergence}.
The advantage of such methods is that they can be essentially explicit, requiring only the
solution of a scalar nonlinear equation at each step.
The relaxation approach
has been extended in order to conserve multiple invariants \cite{biswas2023multiple,biswas2024accurate}, although
the resulting method is more costly and less robust, occasionally requiring the
use of small timesteps.

The time discretization we use in the present work is an extension of
our previous work \cite{ranocha2025high}, in which we combine orthogonal
projection with relaxation to conserve mass, momentum, and energy for the BBM, KdV,
and NLS equations.

\subsection{Full discretizations}

From the large body of literature on structure-preserving methods for
the BBM, KdV, and NLS equations, we are only aware of three (rather recent) papers
developing fully-discrete methods conserving mass, momentum, and energy
for one of these equations:
Zheng and Xu \cite{zheng2024invariants} use a fully implicit LDG method
with Lagrange multipliers and a spectral deferred correction approach for the
KdV equation,
and Akrivis et al.\ \cite{akrivis2025high} use a fully-implicit space-time
finite element method with Lagrange multipliers for the NLS equation.
After the preprint of this manuscript was published on arxiv.org, the paper \cite{shi2026class} presenting a fully-discrete method conserving mass, momentum, and energy for the KdV equation using Fourier Galerkin in space and the projection method based on discrete gradients of \cite{dahlby2011preserving} in time applied to exponential integrators was published.
In contrast to these methods, our schemes are less implicit and more flexible
in terms of the choice of the temporal discretization.
Instead of designing a conservative spatial scheme, for one-dimensional
problems with smooth enough solutions one can simply use a highly resolved
Fourier collocation method that is essentially exact (up to machine precision).
Alvarez et al.\  \cite{alvarez2010multi} combined this with
projection methods in time applied to a singly-diagonally-implicit Runge-Kutta (SDIRK) method to
conserve mass, momentum, and energy for the KdV equation.

We cannot provide a complete overview of all papers on structure-preserving
methods for the BBM, KdV, and NLS equations here. However, we will briefly
summarize some related literature and apologize for any omissions.

For the NLS equation, the popular split-step Fourier pseudospectral
method provides conservation of mass only (see e.g. \cite{bao2002time}).
The linear invariant of the BBM and KdV equations is conserved by most methods we are aware of.

There are some general methodological developments that have been applied to several equations.
Discrete variational derivative methods can be used to conserve two invariants of BBM, KdV, and NLS \cite{koide2009nonlinear,furihata2010discrete}.
Frasca-Caccia and Hydon have designed bespoke finite difference methods conserving local forms of the conservation laws for two invariants for BBM, KdV, and NLS \cite{frasca2021numerical,frasca2020simple}.

There are many methods conserving two invariants at the fully discrete level.
For the NLS equation, some mass- and energy-conserving methods are studied in
\cite{delfour1981finite,sanz1984methods,akrivis1991fully,henning2017crank,besse2004relaxation,besse2021energy,cui2021mass,bai2024high,biswas2024accurate,ranocha2025high};
methods conserving the mass and the Hamiltonian structure are analyzed in \cite{cano2006conserved}.
SBP operators can be combined with relaxation to conserve mass and either momentum or energy for the BBM and KdV equations \cite{ranocha2021broad,linders2023resolving,ranocha2020relaxationHamiltonian} and their hyperbolic approximations \cite{bleecke2025asymptotic,biswas2025traveling}.
The SAV method can be used to conserve mass,
momentum, and a modified energy for generalized KdV equations \cite{yang2022arbitrarily}.

Fourier Galerkin space discretization has recently been combined with a symplectic Runge-Kutta time discretization
to achieve fully-discrete conservation of the linear and quadratic invariants \cite{dougalis2022high}.
Andrews and Farrell \cite{andrews2025conservative} have developed fully
implicit time integration methods able to conserve multiple invariants
(up to quadrature and solver tolerances), and applied the procedure to
conserve the energy of the BBM equation.

There are also techniques that directly develop a conservative space-time
discretization; these include Grant's method
\cite{grant2015bespoke,frasca2021numerical}, the discrete variational derivative method
\cite{furihata2010discrete}, and the discrete multiplier method
\cite{wan2016multiplier,schulz2025minimal}.  In principle these approaches could
be applied to conserve an arbitrary number of invariants\footnote{Such extensions may involve techniques proposed for discrete gradients in \cite{dahlby2011preserving}.}, but in practice
conserving more than two leads to major difficulties, and has not been demonstrated.

\subsection{Contributions and outline}
As we can see from the foregoing, most structure-preserving methods
have one or more of the following drawbacks: they conserve only one or
two invariants, they are limited to second order, and they require the
solution of large systems of algebraic equations. Here we provide methods
that combine the following advantages:
\begin{itemize}
    \item conservation of mass, momentum, and energy;
    \item arbitrary order in space and time;
    \item only a scalar equation must be solved at each step.
\end{itemize}
An additional advantage of the present approach is that we can use any baseline
method, e.g., an IMEX method for the KdV and NLS equations to handle the stiff
linear terms efficiently, and an explicit method for the non-stiff BBM equation.

The most important restrictions of these new methods are that they
require periodic boundary conditions and they provide only global
(not local) conservation.

In Section~\ref{sec:spatial_semidiscretizations} we focus on spatial
discretization, showing that Fourier Galerkin methods provide the desired
conservation properties.
We also point out some crucial subtleties to obtain the desired results in
(acceptably efficient) implementations, and demonstrate the semidiscrete
conservation numerically.
In Section~\ref{sec:time_discretizations} we introduce our conservative time
discretization method, which combines orthogonal projection with relaxation and
is an extension of that proposed in our recent work \cite{ranocha2025high}.
Fully-discrete conservation is demonstrated through numerical experiments.
In Section~\ref{sec:error-growth} we study the long-term error behavior of
our conservative methods compared with methods that conserve fewer invariants.
In Section~\ref{sec:performance_comparisons} we perform a practical comparison
of computational efficiency between our proposed methods and some recent methods
from the literature.
In Section~\ref{sec:hyperbolization_nls} we apply the same ideas to a first-order hyperbolic approximation
of NLS.  Some conclusions and future directions are discussed in Section~\ref{sec:summary}.

\section{Spatial semidiscretizations}
\label{sec:spatial_semidiscretizations}

We use Fourier Galerkin methods to discretize the PDEs in space.
Let $T_k$ be the space of real-valued trigonometric polynomials of degree
at most $k$ and let $P$ be the $L^2$ projection\footnote{We omit
the index $k$ from $P$ since we will only use a fixed $k$ in each equation,
not multiple spaces.} onto $T_k$. We denote the $L^2$ inner product
on the spatial domain by $\langle \cdot, \cdot \rangle$.
Next, we will introduce the resulting semidiscretizations and prove that
they conserve the mass, momentum, and energy for each equation.

\subsection{Benjamin-Bona-Mahony equation}

The Fourier Galerkin semidiscretization of the BBM equation \eqref{eq:bbm}
is given by
\begin{equation}
\label{eq:bbm_semi}
  \partial_t u = -(\I - \partial_x^2)^{-1} \partial_x P \frac{u^2}{2}.
\end{equation}

\begin{theorem}
\label{thm:bbm_semi}
  The semidiscretization \eqref{eq:bbm_semi} of the BBM equation \eqref{eq:bbm}
  conserves the mass, momentum, and energy \eqref{eq:bbm_invariants}.
\end{theorem}
\begin{proof}
  The semidiscretization \eqref{eq:bbm_semi} conserves the total mass
  $\mass = \int u \dif x$,
  since
  \begin{equation}
    \partial_t \mass
    = \left\langle 1, \partial_t u \right\rangle
    = -\left\langle 1, (\I - \partial_x^2)^{-1} \partial_x P \frac{u^2}{2} \right\rangle
    = \left\langle (\I - \partial_x^2)^{-1} \partial_x 1, P \frac{u^2}{2} \right\rangle
    = 0,
  \end{equation}
  where we used the skew-symmetry of $(\I - \partial_x^2)^{-1} \partial_x$
  in the second-to-last step and $\partial_x 1 = 0$ in the last step.
  Similarly, it conserves the momentum
  $\momentum = \int (u^2 + u_x^2) / 2 \dif x = \int u (\I - \partial_x^2) u / 2 \dif x$,
  since
  \begin{equation}
  \begin{aligned}
    2 \partial_t \momentum
    = 2 \left\langle (\I - \partial_x^2) u, \partial_t u \right\rangle
    &= -\left\langle (\I - \partial_x^2) u, (\I - \partial_x^2)^{-1} \partial_x P u^2 \right\rangle
    \\
    &= \left\langle \partial_x u, P u^2 \right\rangle
    = \left\langle \partial_x u, u^2 \right\rangle
    = 0,
  \end{aligned}
  \end{equation}
  where we used that $\partial_x u \in T_k$ in the second-to-last step
  so that we can use the chain rule in the last step. Moreover, the total
  energy $\energy = \int u^3 / 6 \dif x$ is also conserved, since
  \begin{equation}
    4 \partial_t \energy
    = 2 \left\langle u^2, \partial_t u \right\rangle
    = -\left\langle u^2, (\I - \partial_x^2)^{-1} \partial_x P u^2 \right\rangle
    = -\left\langle P u^2, (\I - \partial_x^2)^{-1} \partial_x P u^2 \right\rangle
    = 0,
  \end{equation}
  where we used that the semidiscrete rate of change $\partial_t u \in T_k$
  in the second-to-last step and the skew-symmetry of
  $(\I - \partial_x^2)^{-1} \partial_x$ in the last step.
\end{proof}

\subsection{Korteweg-de Vries equation}

The Fourier Galerkin semidiscretization of the KdV equation \eqref{eq:kdv}
is given by
\begin{equation}
\label{eq:kdv_semi}
  \partial_t u = -\partial_x P \frac{u^2}{2} - \partial_x^3 u.
\end{equation}

\begin{theorem}[Maday and Quarteroni \cite{maday1988error}]
\label{thm:kdv_semi}
  The semidiscretization \eqref{eq:kdv_semi} of the KdV equation \eqref{eq:kdv}
  conserves the mass, momentum, and energy \eqref{eq:kdv_invariants}.
\end{theorem}
\begin{proof}
  The semidiscretization \eqref{eq:kdv_semi} conserves the total mass
  $\mass = \int u \dif x$, since
  \begin{equation}
    \partial_t \mass
    = \left\langle 1, \partial_t u \right\rangle
    = -\left\langle 1, \partial_x P \frac{u^2}{2} \right\rangle - \left\langle 1, \partial_x^3 u \right\rangle
    = \left\langle \partial_x 1, P \frac{u^2}{2} \right\rangle + \left\langle \partial_x^3 1, u \right\rangle
    = 0.
  \end{equation}
  Similarly, it conserves the momentum $\momentum = \int u^2 / 2 \dif x$, since
  \begin{equation}
    \partial_t \momentum
    = \left\langle u, \partial_t u \right\rangle
    = -\left\langle u, \partial_x P \frac{u^2}{2} \right\rangle - \left\langle u, \partial_x^3 u \right\rangle
    = \left\langle \partial_x u, \frac{u^2}{2} \right\rangle + \left\langle \partial_x^3 u, u \right\rangle
    = 0,
  \end{equation}
  where we used that $\partial_x u \in T_k$ in the second-to-last step
  so that we can use the chain rule in the last step. Moreover, the total
  energy
  $\energy = \int (u_x^2 / 2 - u^3 / 6) \dif x = -\int (u \partial_x^2 u / 2 + u^3 / 6) \dif x$
  is also conserved, since
  \begin{equation}
  \begin{aligned}
    \partial_t \energy
    &=
    - \left\langle \partial_x^2 u + \frac{u^2}{2}, \partial_t u \right\rangle
    \\
    &=
    \left\langle \partial_x^2 u, \partial_x P \frac{u^2}{2} \right\rangle
    + \left\langle \frac{u^2}{2}, \partial_x P \frac{u^2}{2} \right\rangle
    + \left\langle \partial_x^2 u, \partial_x^3 u \right\rangle
    + \left\langle \frac{u^2}{2}, \partial_x^3 u \right\rangle
    \\
    &=
    - \left\langle \partial_x^3 u, P \frac{u^2}{2} \right\rangle
    + \left\langle P \frac{u^2}{2}, \partial_x P \frac{u^2}{2} \right\rangle
    + 0
    + \left\langle P \frac{u^2}{2}, \partial_x^3 u \right\rangle
    =
    0,
  \end{aligned}
  \end{equation}
  where we used the skew-symmetry of $\partial_x$.
\end{proof}

\begin{remark}
  The Fourier Galerkin semidiscretization \eqref{eq:kdv_semi} of the
  KdV equation \eqref{eq:kdv} does not conserve the fourth invariant
  given by
  \begin{multline}
    \partial_t \biggl(
      \frac{1}{24} u^4
      - \frac{1}{2} u u_x^2
      + \frac{3}{10} u_{xx}^2
    \biggr)
    +
    \partial_x \biggl(
      \frac{1}{30} u^5
      + \frac{1}{6} u^3 u_{xx}
      - \frac{3}{4} u^2 u_x^2
      - u u_x u_{xxx}
    \\
      + \frac{4}{5} u u_{xx}^2
      + \frac{1}{2} u_x^2 u_{xx}
      + \frac{3}{5} u_{xx} u_{xxxx}
      - \frac{3}{10} u_{xxx}^2
    \biggr)
    = 0.
  \end{multline}
  We did not check whether other higher-order invariants are conserved.
\end{remark}

\subsection{Nonlinear Schr{\"o}dinger equation}

To formulate the semidiscretization of the NLS equation \eqref{eq:nls-u},
we rewrite it as a system for the real and imaginary parts $u = v + \i w$:
\begin{equation}
\label{eq:nls-vw}
\begin{aligned}
  v_t + w_{xx} + \beta \bigl( v^2 + w^2 \bigr) w &= 0,
  \\
  w_t - v_{xx} - \beta \bigl( v^2 + w^2 \bigr) v &= 0.
\end{aligned}
\end{equation}
The invariants \eqref{eq:nls-u_invariants} can be rewritten as
\begin{equation}
\label{eq:nls-vw_invariants}
\begin{gathered}
  \mass = \int \left( v^2 + w^2 \right) \dif x,
  \qquad
  \momentum = \int \left( v w_x - w v_x \right) \dif x = 2 \int v w_x \dif x,
  \\
  \energy = \int \left( v_x^2 + w_x^2 - \frac{\beta}{2} \left( v^2 + w^2 \right)^2 \right) \dif x.
\end{gathered}
\end{equation}
The Fourier Galerkin semidiscretization of \eqref{eq:nls-vw} is given by
\begin{equation}
\label{eq:nls-vw_semi}
\begin{aligned}
  \partial_t v &= -w_{xx} - \beta P \bigl( v^2 + w^2 \bigr) w,
  \\
  \partial_t w &=  v_{xx} + \beta P \bigl( v^2 + w^2 \bigr) v.
\end{aligned}
\end{equation}

\begin{theorem}
\label{thm:nls-vw_semi}
  The semidiscretization \eqref{eq:nls-vw_semi} of the NLS equation
  \eqref{eq:nls-vw} conserves the mass, momentum, and energy
  \eqref{eq:nls-vw_invariants}.
\end{theorem}
\begin{proof}
  The semidiscretization \eqref{eq:nls-vw_semi} conserves the total mass
  $\mass = \int (v^2 + w^2) \dif x$, since
  \begin{equation}
  \begin{aligned}
    \partial_t \mass
    &= 2 \left\langle v, \partial_t v \right\rangle + 2 \left\langle w, \partial_t w \right\rangle
    \\
    &= - 2 \left\langle v, w_{xx} + \beta P (v^2 + w^2) w \right\rangle
       + 2 \left\langle w, v_{xx} + \beta P (v^2 + w^2) v \right\rangle
    = 0,
  \end{aligned}
  \end{equation}
  where we used the symmetry of $\partial_x^2$ and exactness of the $L^2$
  projection $P$ in the last step.
  Similarly, the semidiscretization conserves the momentum
  $\momentum = 2 \int v w_x \dif x$, since
  \begin{equation}
  \begin{aligned}
    \partial_t \momentum
    &=   2 \left\langle w_x, \partial_t v \right\rangle
       - 2 \left\langle v_x, \partial_t w \right\rangle
    \\
    &= - 2 \left\langle w_x, w_{xx} + \beta P \bigl( v^2 + w^2 \bigr) w \right\rangle
       - 2 \left\langle v_x, v_{xx} + \beta P \bigl( v^2 + w^2 \bigr) v \right\rangle
    \\
    &= -2 \beta \int (v^2 + w^2) (w w_x + v v_x) \dif x
    = -\frac{1}{2} \beta \int \partial_x (v^2 + w^2)^2 \dif x
    = 0,
  \end{aligned}
  \end{equation}
  where we used the anti-symmetry of $\partial_x$ and exactness of the $L^2$
  projection $P$.
  Finally, the total energy
  $\energy = \int (v_x^2 + w_x^2 - \beta (v^2 + w^2)^2 / 2) \dif x = -\int (v \partial_x^2 v + w \partial_x^2 w + \beta (v^2 + w^2)^2 / 2) \dif x$
  is also conserved, since
  \begin{equation}
  \begin{aligned}
    \partial_t \energy
    &= - 2 \left\langle v_{xx} + \beta (v^2 + w^2) v, \partial_t v \right\rangle
       - 2 \left\langle w_{xx} + \beta (v^2 + w^2) w, \partial_t w \right\rangle
    \\
    &=   2 \left\langle v_{xx} + \beta P (v^2 + w^2) v, w_{xx} + \beta P \bigl( v^2 + w^2 \bigr) w \right\rangle
    \\&\quad
       - 2 \left\langle w_{xx} + \beta P (v^2 + w^2) w, v_{xx} + \beta P \bigl( v^2 + w^2 \bigr) v \right\rangle
    = 0,
  \end{aligned}
  \end{equation}
  where we used again $\partial_t v, \partial_t w \in T_k$ and the exactness
  of the $L^2$ projection $P$.
\end{proof}

\subsection{Implementation notes}

Since we work with real-valued functions, we use the forward/backward
real FFT (in FFTW \cite{frigo2005design})
to map between spatial values and modal coefficients. While it is often
convenient to use the spatial coordinates as the primary representation
(which can be visualized directly), we cannot do so if an even number of nodes
is used; in this case, the backward real FFT drops the imaginary part of the
highest mode (Nyquist frequency) to ensure that the output is real-valued.
This loss of information cannot be recovered later and destroys the structures
we used to prove conservation. Thus, we use the modal coefficients as primary
variables and compute the spatial derivatives by multiplying all modes by
the appropriate (powers of the) imaginary unit and wave number. Please note
that this differs from the common practice to set the Nyquist frequency of
odd-derivative operators to zero for an even number of nodes
\cite{johnson2011notes}. However, it is required to obtain the desired
conservation of the mass, momentum, and energy for an even number of nodes.

To compute the exact $L^2$ projection $P$, we use a classical de-aliasing
strategy. Assume we have a polynomial nonlinearity of degree $p$ and
$N$ spatial nodes to represent $u$. First, we use the modal coefficients
of $u$ and extend them by zero to obtain the modal coefficients of the
representation of $u$ with $M$ nodes. Then, we compute the nodal values on
these $M$ nodes, compute the nonlinearity in physical space,
compute the modal coefficients of the result on $M$ nodes, and truncate them
back to the modal coefficients corresponding to $N$ nodes. This ensures that
we obtain the exact $L^2$ projection of the degree $p$ nonlinearity if
$M > (p + 1) N / 2$,
see, e.g., \cite[Section~4.3.2]{kopriva2009implementing} for a detailed
description of the case $p = 2$ and
\cite{jones1996pseudo,derevyanko2008n_plus_1_over_2} for general $p$.
For quadratic nonlinearities with $p = 2$, this is the well-known
$3/2$-rule \cite{orszag1971elimination}.

To improve the efficiency, we compute a minimal number $M_{\mathrm{min}}$ of
nodes to achieve de-aliasing based on the rules above. Then, we choose
$M \ge M_{\mathrm{min}}$ as the smallest integer
with prime factors in $\{2, 3, 5, 7\}$.
To further improve the performance of the implementations, we use several larger
grids (with $M > N$ nodes) to compute the nonlinearities in physical space, e.g.,
based on a factor $3/2$ for the quadratic nonlinearities $u u_x$ in the BBM
and KdV equations and a factor $2$ for the cubic nonlinearities in the energies
of these equations. We initialize the numerical solution by sampling the
initial condition at $N$ nodes and computing the corresponding modal
coefficients.

We implemented all methods in Julia \cite{bezanson2017julia}.
We use FFTW.jl \cite{frigo2005design} wrapped in SummationByPartsOperators.jl
\cite{ranocha2021sbp} for the Fourier Galerkin methods. The visualizations
are created using Makie.jl \cite{danisch2021makie}.
All code and data required to reproduce the numerical results are available
online in our reproducibility repository \cite{ranocha2025conservingRepro}.

\subsection{Numerical verification of the semidiscrete invariant conservation}
\label{sec:semidiscrete_conservation}

Next, we verify the conservation of the mass, momentum, and energy
for the semidiscretizations of the BBM, KdV, and NLS equations numerically.
While explicit solitary wave (soliton) solutions are available for all three
equations, a single solitary wave is often not challenging enough to demonstrate
conservation of multiple invariants. Thus, we use setups with two interacting
waves for each equation.

The BBM equation \eqref{eq:bbm} has solitary wave solutions
\begin{equation}
\label{eq:bbm_solitary_wave}
  u_{x_0, c}(t, x) = 1 + A \cosh\bigl(k (x - x_0 - c t) \bigr)^{-2},
  \qquad
  A = 3 (c - 1),
  \quad
  k = \frac{1}{2} \sqrt{1 - 1 / c}.
\end{equation}
We initialize the numerical solution as
\begin{equation}
  u(0, x) = u_{-20, 1.3}(0, x) + u_{20, 1.2}(0, x) - 1
\end{equation}
in the domain $[-100, 100]$ with periodic boundary conditions and choose a time span
of $[0, 400]$ to ensure that the two waves interact with each other.

The KdV equation \eqref{eq:kdv} has the two-soliton solution
\cite{hirota1971exact,hietarinta2007introduction}
\begin{equation}
  u(t, x) = 12 \partial_x^2 \log F,
  \qquad
  F = 1 + \e^{\eta_1} + \e^{\eta_2} + a_{12} \e^{\eta_1 + \eta_2},
  \quad
  \eta_i = k_i (x - x_{i, 0}) - k_i^3 t.
\end{equation}
We choose the parameters
\begin{equation}
  k_1 = 0.75, \quad k_2 = 0.5,
  \qquad
  x_{1, 0} = -50, \quad x_{2, 0} = 50,
\end{equation}
and the spatial domain $[-200, 200]$ with periodic boundary conditions.
The time span $[0, 350]$ ensures that the two waves interact. We use
Enzyme.jl \cite{moses2020instead,moses2021reverse} to compute the
second-derivative via automatic/algorithmic differentiation (AD).

For the NLS equation \eqref{eq:nls-u}, we use the same two-soliton solution
as in \cite{biswas2024accurate,ranocha2025high}; we choose the spatial domain
$[-35, 35]$ with periodic boundary conditions and the time span $[0, 10]$.

\begin{figure}[!htb]
\centering
  \includegraphics[width=\textwidth]{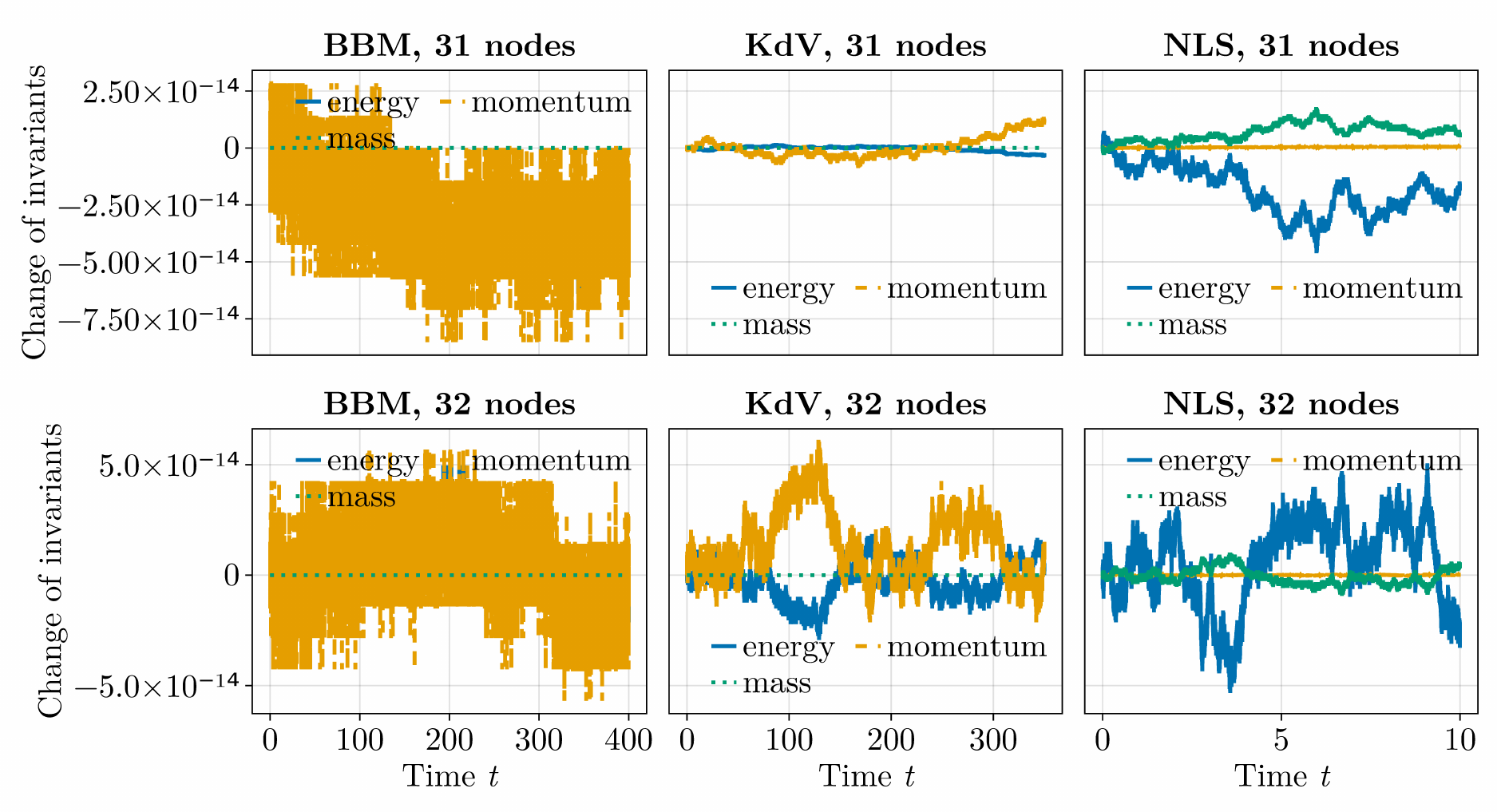}
  \caption{Change of invariants over time for the Fourier Galerkin semidiscretizations
           of the BBM, KdV, and NLS equations with two interacting solitary waves.
           The time integration is performed with the fifth-order method of
           \cite{kennedy2019higher} with $\Delta t = \num{5e-3}$ for the BBM,
           $\Delta t = \num{1e-2}$ for the KdV, and
           $\Delta t = \num{1e-4}$ for the NLS equation.}
  \label{fig:semidiscrete_conservation}
\end{figure}

For all three equations, we use coarse meshes with $N \in \{31, 32\}$ nodes
and sufficiently small time step sizes to ensure that the errors in time are
negligible. The results are shown in Figure~\ref{fig:semidiscrete_conservation}.
As expected, the Fourier Galerkin semidiscretizations conserve all three invariants
up to the precision of the time integration method (which is close to machine
accuracy due to the choice of sufficiently small time step sizes $\Delta t$).

\subsection{Numerical verification of the convergence in space}
\label{sec:convergence_in_space}

To verify the convergence of the Fourier Galerkin semidiscretizations in space, we use the same setups as in Section~\ref{sec:semidiscrete_conservation}.
The only difference is the BBM equation, where we choose a strictly negative solitary wave with celerity $c = -5$ in \eqref{eq:bbm_solitary_wave} in the periodic domain $[-50, 50]$ to obtain a large-amplitude solution that is numerically more challenging to approximate.
We use sufficiently small time steps with the fifth-order method of \cite{kennedy2019higher} to ensure that the errors in time are negligible.
We compute the relative $L^2$ errors at the final time (i.e., the discrete $L^2$ error divided by the discrete $L^2$ norm of the exact solution).
The results shown in Figure~\ref{fig:convergence_in_space} demonstrate the expected exponential convergence of the Fourier Galerkin semidiscretizations in space for all three equations.

We also considered negative celerities closer to zero, which can influence the orbital stability \cite{nguyen2009stability}. The results were similar to the ones for $c = -5$; in particular, we obtained the same convergence behavior. Thus, we do not show them here in detail.

\begin{figure}[!htb]
\centering
  \includegraphics[width=\textwidth]{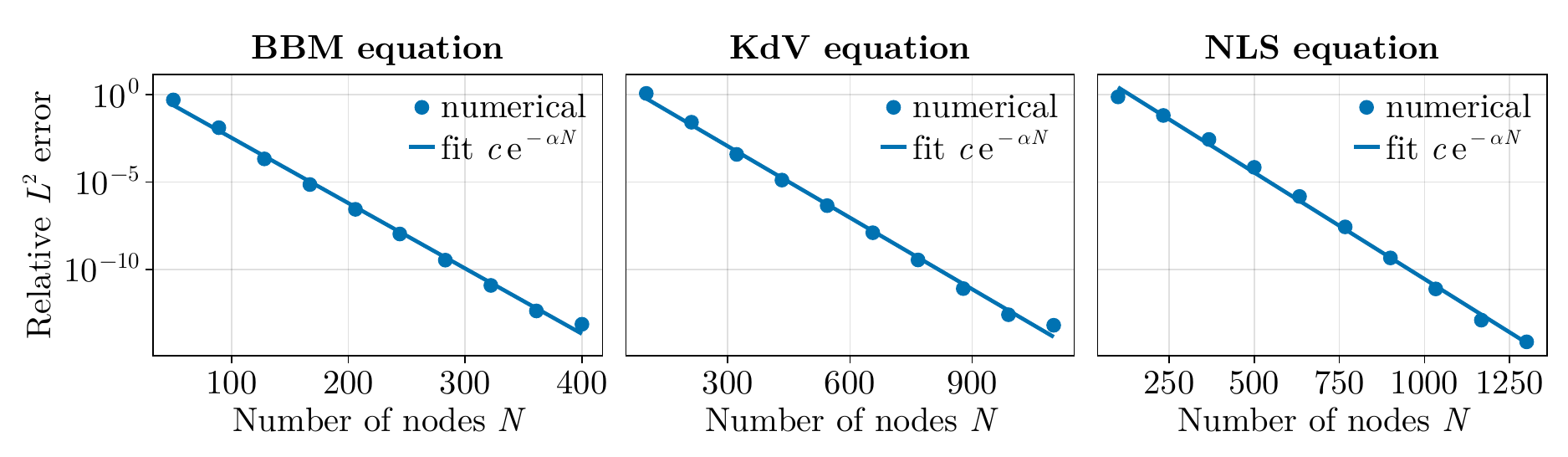}
  \caption{Relative $L^2$ errors at the final time for the Fourier Galerkin semidiscretizations
           of the BBM, KdV, and NLS equations.
           The time integration is performed with the fifth-order method of
           \cite{kennedy2019higher} with
           $2^{14}$ time steps in $[0, 10]$ for BBM,
           $2^{15}$ time steps in $[0, 100]$ for KdV, and
           $2^{14}$ time steps in $[0, 1]$ for NLS.}
  \label{fig:convergence_in_space}
\end{figure}

\section{Time discretizations}
\label{sec:time_discretizations}

General linear methods like Runge-Kutta methods and
linear multistep methods typically conserve only the linear invariants,
e.g., the total mass $\mass$ of the BBM and KdV equations. While there
are some special combinations of methods and problems where simple explicit
schemes conserve a nonlinear invariant \cite[Section~5]{ranocha2020energy},
this cannot be expected in general. Quadratic invariants like the momentum
$\momentum$ for the BBM, KdV, and NLS equations are conserved by symplectic
Runge-Kutta methods, which are necessarily fully implicit
\cite{hairer2006geometric}.

We want to avoid fully implicit methods due to the high computational costs
of solving large nonlinear systems. Thus, we use IMEX
additive Runge-Kutta (ARK) methods \cite{ascher1997implicit,kennedy2019higher},
which treat the linear stiff terms of the KdV and NLS equations implicitly
and all nonlinear terms explicitly (for the BBM equation, we only use the
explicit parts of the ARK methods). Since we use Fourier methods in space,
the resulting linear systems can be solved efficiently in modal space.

To enforce conservation of all invariants, we use the quadratic-preserving
relaxation method from \cite{ranocha2025high}. It combines an orthogonal
projection (see, e.g., \cite{calvo2010projection} or \cite[Section~IV.4]{hairer2006geometric}) with
relaxation (see, e.g., \cite{ranocha2020general,ranocha2020relaxation}).
For an ODE
\begin{equation}
\label{eq:ode}
  u'(t) = f\bigl( u(t) \bigr)
\end{equation}
with invariant (first integral) $\eta$ satisfying
$\forall u\colon \eta'(u) f(u) = 0$,
the quadratic-preserving relaxation method \cite{ranocha2025high} performs
the following steps:
\begin{itemize}
  \item Given $u^n \approx u(t^n)$, compute a provisional value
        $\widetilde{u}^{n+1} \approx u(\widetilde{t}^{n+1})$
        using a baseline time integration method, e.g., an ARK method.
  \item Project the baseline result onto the manifold defined by a quadratic
        invariant using the projection operator $\pi$, i.e., compute
        $\widehat{u}^{n+1} = \pi\bigl( \widetilde{u}^{n+1} \bigr)$.
  \item Search for a solution conserving the additional invariant $\eta$
        along the (approximate) geodesic line connecting $u^{n}$ and
        $\widehat{u}^{n+1}$, i.e., solve the scalar equation
        \begin{subequations} \label{qpr}
        \begin{equation} \label{gamma}
          \eta\Bigl( \pi \bigl( u^n + \gamma \bigl( \hat{u}^{n+1} - u^n \bigr) \bigr) \Bigr)
          =
          \eta(u^n)
        \end{equation}
        for the scalar relaxation parameter $\gamma$.
  \item Continue the numerical time integration with
        \begin{equation}
          u^{n+1} = \pi \bigl( u^n + \gamma \bigl( \hat{u}^{n+1} - u^n \bigr) \bigr)
          \approx u\bigl( t^{n+1} \bigr),
          \qquad
          t^{n+1} = t^n + \gamma \Delta t,
        \end{equation}
        \end{subequations}
        instead of $\tilde{u}^{n+1}$ and $\tilde{t}^{n+1}$.
\end{itemize}
For the nonlinear Schr{\"o}dinger equation considered in \cite{ranocha2025high},
the projection operator conserving the mass from one step to the next is given by
\begin{equation}
  \pi\bigl( \tilde{u}^{n+1} \bigr)
  =
  \sqrt{\frac{\mass\bigl( u^{n} \bigr)}
             {\mass\bigl( \tilde{u}^{n+1} \bigr)}} \tilde{u}^{n+1}.
\end{equation}
By construction, the invariant enforced by $\pi$ and the invariant $\eta$
are conserved. In \cite{ranocha2025high}, $\eta$ is chosen as the energy
$\energy$ of the NLS equation.

We generalize this approach by choosing the projection operator $\pi$
to conserve both the mass $\mass$ and the momentum $\momentum$. This allows us
to construct relaxation methods conserving all three invariants of the
BBM, KdV, and NLS equations by choosing $\eta$ as the energy $\energy$.
From \cite{ranocha2025high}, we have
\begin{theorem}
\label{thm:geodesic_relaxation}
  Assume that the ODE \eqref{eq:ode} has the invariants $(\mass, \momentum,
  \energy)$ and that $\pi$ is a projection operator onto the manifold defined
  by the two constraints $\mass(u)=\mass(u^n)$ and $\momentum(u)=\momentum(u^n)$.
  Let $\eta=\energy$ and assume that the baseline one-step method is of order
  $p \ge 2$ and that
  \begin{equation}
  \label{eq:nondegeneracy_condition}
    \eta'\bigl( u^n \bigr) \pi'\bigl( u^n \bigr) f'\bigl( u^n \bigr) f\bigl( u^n \bigr)
    \ne 0.
  \end{equation}
  Then, the generalized quadratic-preserving relaxation method \eqref{qpr} is well-defined
  for sufficiently small time step sizes $\Delta t$; there is a
  unique solution of \eqref{gamma} with $\gamma = 1 + \mathcal{O}(\Delta t^{p-1})$ and
  the resulting order of accuracy is at least $p$ (when measuring the
  error at the relaxed time $t^{n+1} = t^n + \gamma \Delta t$).
  Moreover, all three invariants are conserved.
\end{theorem}

The non-degeneracy condition \eqref{eq:nondegeneracy_condition} is a
generalization of similar conditions for standard relaxation methods
and is discussed further in \cite{ranocha2025high}. For example, it ensures
that $u^n$ is not a steady state of the ODE.

In practice, we solve the scalar nonlinear equation for $\gamma$ using
the method of \cite{klement2014using} implemented in
SimpleNonlinearSolve.jl \cite{pal2026nonlinearsolve}.

\subsection{Implementation of the projection operators}

For the BBM and KdV equations, the mass $\mass$ is a linear invariant, while
the momentum $\momentum$ is a quadratic invariant and induces a
norm. Using this norm to measure distances, the orthogonal projection
onto the manifold defined by constant mass and momentum is given by
\begin{equation}
  \pi(u^{n+1}) = \overline{u} + \sqrt{\frac{\momentum(u^n) - \momentum(\overline{u})}{\momentum(u^{n+1} - \overline{u})}} (u^{n+1} - \overline{u}),
\end{equation}
where
\begin{equation}
  \overline{u} = \frac{\mass(u^n)}{\mass(1)} = \frac{\mass(u^{n+1})}{\mass(1)}
\end{equation}
is the mean value of $u^{n}$ (and $u^{n+1}$, since the baseline methods
conserve the total mass due to its linearity).

For the NLS equation, both the total mass $\mass$ and the total momentum
$\momentum$ are quadratic invariants. Since even the projection onto an
ellipsoid is not completely straightforward (since it involves solving a
quartic equation, for which no simple closed-form solution exists), we use
a simplified projection method using the gradient of the momentum at the
current solution $u^n = v^n + \i w^n$. This leads to the ansatz
\begin{equation}
  \pi
  \begin{pmatrix}
    v \\
    w
  \end{pmatrix}
  =
  \lambda \begin{pmatrix}
    v \\
    w
  \end{pmatrix}
  +
  \mu \begin{pmatrix}
    w_x \\
    - v_x
  \end{pmatrix}
\end{equation}
for Lagrange multipliers $\lambda, \mu \in \R$ such that
\begin{equation}
  \mass\bigl( \pi(u) \bigr) = \mass\bigl( u^{n} \bigr)
  \qquad\text{and}\qquad
  \momentum\bigl( \pi(u) \bigr) = \momentum\bigl( u^{n} \bigr).
\end{equation}
The condition on the total mass can be written as
\begin{equation}
\label{eq:nls_mass_condition}
\begin{aligned}
  \mass(u^n)
  &=
  \mass\bigl( \pi(u) \bigr)
  =
  \int \left( (\lambda v + \mu w_x)^2 + (\lambda w - \mu v_x)^2 \right)
  \\
  &=
  \lambda^2 \int (v^2 + w^2)
  + 2 \lambda \mu \int (v w_x - w v_x)
  + \mu^2 \int (v_x^2 + w_x^2)
  \\
  &=
  \lambda^2 \mass(v, w)
  + 2 \lambda \mu \momentum(v, w)
  + \mu^2 \mass(v_x, w_x).
\end{aligned}
\end{equation}
The condition on the total momentum can be written as
\begin{equation}
\label{eq:nls_momentum_condition}
\begin{aligned}
  \momentum(u^n)
  &=
  \momentum\bigl( \pi(u) \bigr)
  =
  2 \int (\lambda v + \mu w_x) (\lambda w_x - \mu v_{xx})
  \\
  &=
  2 \lambda^2 \int v w_x
  + 2 \lambda \mu \int (v_x^2 + w_x^2)
  + 2 \mu^2 \int v_x w_{xx}
  \\
  &=
  \lambda^2 \momentum(v, w)
  + 2 \lambda \mu \mass(v_x, w_x)
  + \mu^2 \momentum(v_x, w_x).
\end{aligned}
\end{equation}
For given $(v, w)$, the conditions \eqref{eq:nls_mass_condition}
and \eqref{eq:nls_momentum_condition} form a system of two quadratic equations
for the two unknowns $\lambda$ and $\mu$. We solve this system using
Newton's method implemented in
SimpleNonlinearSolve.jl \cite{pal2026nonlinearsolve} with initial guess
$\lambda = 1$ and $\mu = 0$.

\subsection{Numerical verification of the fully-discrete invariant conservation}
\label{sec:fully_discrete_conservation}

We verify the conservation of the mass, momentum, and energy for the fully
discrete schemes using the same two-wave setups as in
Section~\ref{sec:semidiscrete_conservation}.

\begin{figure}[!htb]
\centering
  \includegraphics[width=\textwidth]{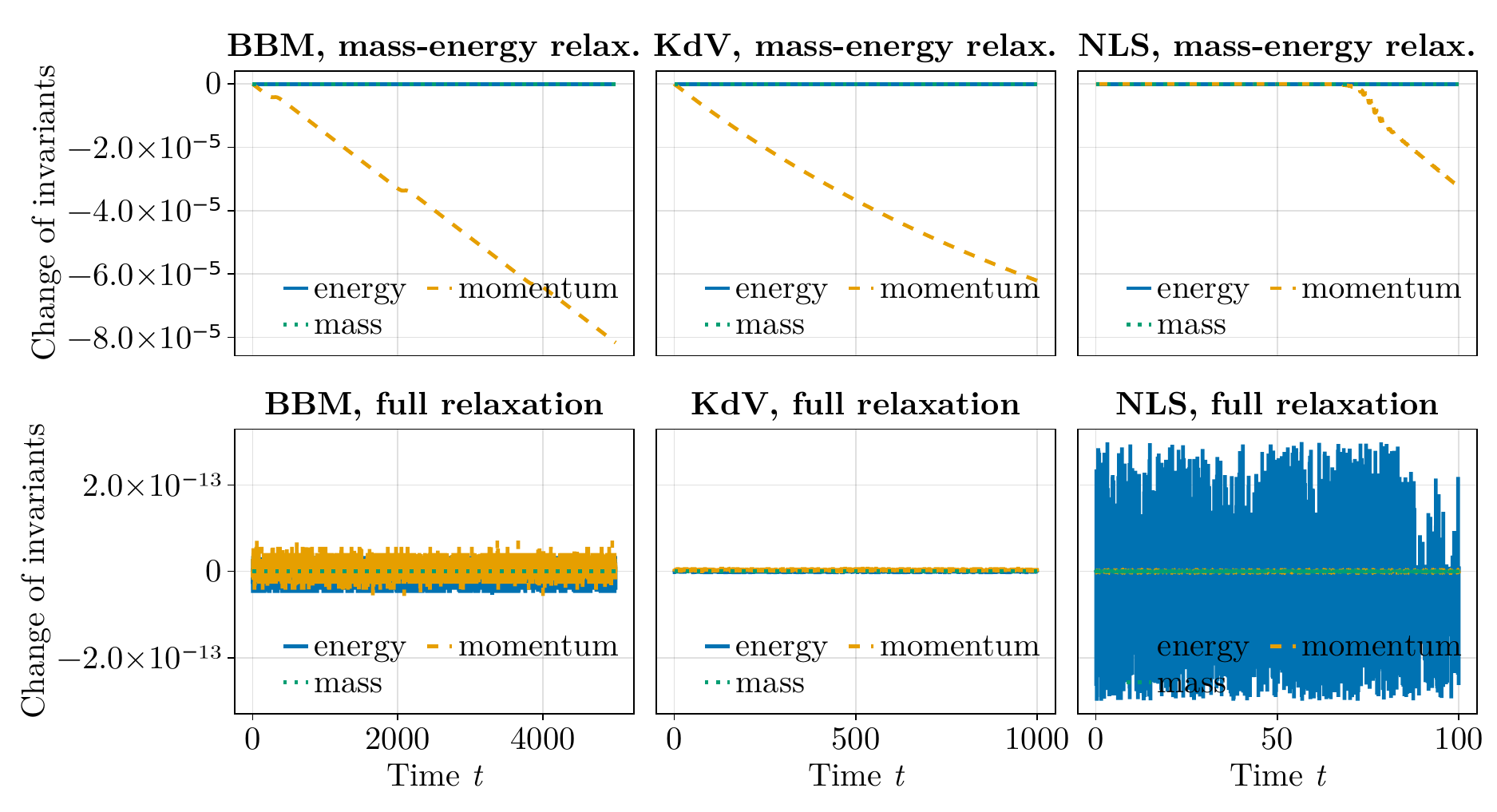}
  \caption{Change of invariants over time for the Fourier Galerkin
           semidiscretizations of the BBM, KdV, and NLS equations with two
           interacting solitary waves for two versions of relaxation.
           The time integration is performed with the fifth-order method of
           \cite{kennedy2019higher} with $\Delta t = 0.5$ for the BBM,
           $\Delta t = 0.1$ for the KdV, and
           $\Delta t = 0.01$ for the NLS equation.}
  \label{fig:fully_discrete_conservation_two_waves}
\end{figure}

We choose $N = 2^8$ nodes in space and use the fifth-order ARK method of
\cite{kennedy2019higher} as the baseline time integration method. Compared to
Section~\ref{sec:semidiscrete_conservation}, we choose bigger time step sizes
$\Delta t$ such that the error in time is not negligible anymore. We choose
final times such that the baseline method performs $10^4$ time steps for each
equation.

The results shown in Figure~\ref{fig:fully_discrete_conservation_two_waves}
confirm the theoretical predictions. In particular, relaxation methods
constructed to conserve the total mass and energy (only) conserve these
invariants but not the total momentum\footnote{For the BBM and KdV equations,
these are the standard relaxation methods that conserve all linear invariants
automatically; for the NLS equation, this is the quadratic-preserving
relaxation method of \cite{ranocha2025high} without the modification to
conserve the momentum.}.
The new quadratic-preserving relaxation methods using the projection operators
to conserve the total mass and momentum conserve all three invariants up to
machine precision.

\begin{figure}[!htb]
\centering
  \includegraphics[width=\textwidth]{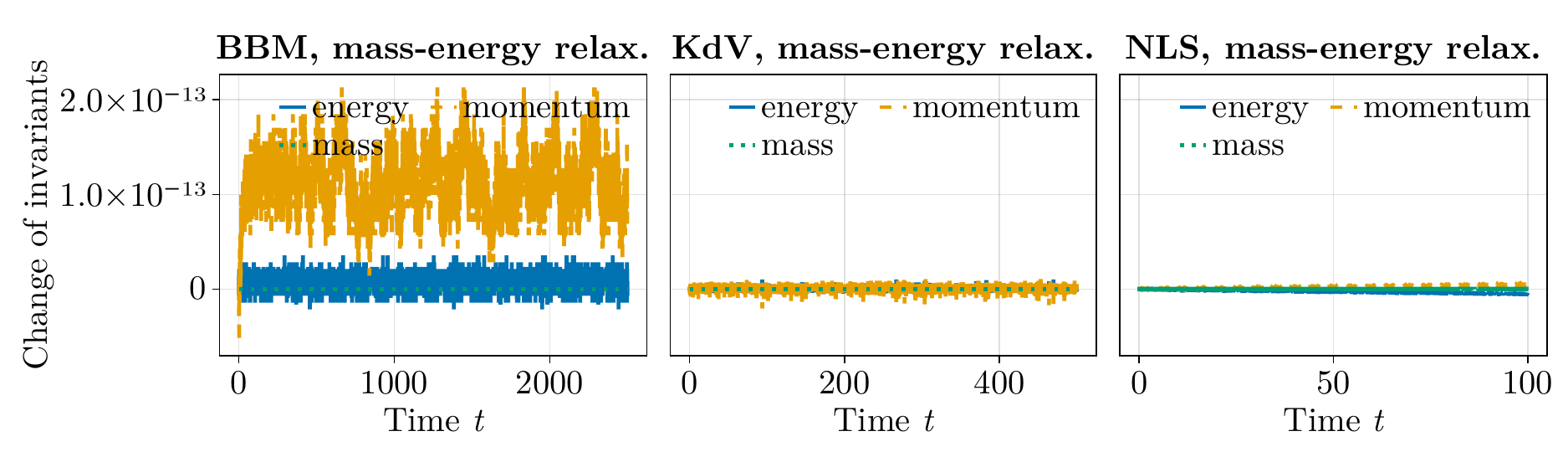}
  \caption{Change of invariants over time for the Fourier Galerkin
           semidiscretizations of the BBM, KdV, and NLS equations with one
           solitary wave with relaxation to enforce conservation of the
           total mass and energy.
           The time integration is performed with the fourth-order method of
           \cite{kennedy2019higher} with $\Delta t = 0.25$ for the BBM,
           $\Delta t = 0.05$ for the KdV, and
           $\Delta t = 0.01$ for the NLS equation.}
  \label{fig:fully_discrete_conservation_one_wave}
\end{figure}

For a single solitary wave, conserving the total mass and energy typically
leads to very good results. In particular, it results in a linear error growth
in time (if the spatial error is negligible) instead of a quadratic error
growth for general time integration methods
\cite{frutos1997accuracy,duran2000numerical,araujo2001error}.
Here, we choose the solitary wave with the larger wave speed from the two-wave
setup for the BBM and KdV equations as well as the single-soliton solution
used in \cite{biswas2024accurate,ranocha2025high} for the NLS equation.

The results shown in Figure~\ref{fig:fully_discrete_conservation_one_wave}
demonstrate that even the mass- and energy-conserving relaxation methods
lead to conservation of the momentum up to small oscillations close to
machine accuracy. Please note that this is a special property when integrating
a single solitary wave and does not hold in more general cases such as the
two-wave setups considered before.

\subsection{Numerical verification of the convergence in time}
\label{sec:convergence_in_time}

To verify the convergence of the time integration methods, we use the same setups as in Section~\ref{sec:convergence_in_time}.
The relative $L^2$ errors at the final time shown in Figure~\ref{fig:convergence_in_time} demonstrate the expected convergence rates of the time integration methods for all three equations.
In all cases, the relaxation method conserving the three invariants $\mass$, $\momentum$, and $\energy$ has a smaller error than the baseline time integration method.

\begin{figure}[!htb]
\centering
  \includegraphics[width=\textwidth]{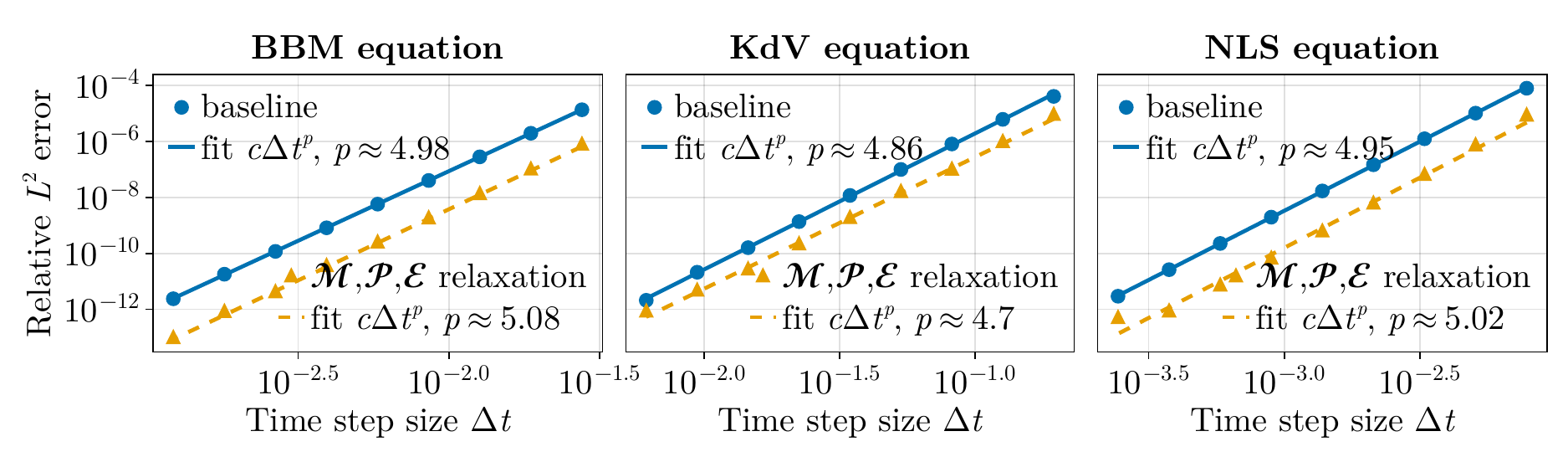}
  \caption{Relative $L^2$ errors at the final time for the Fourier Galerkin semidiscretizations
           of the BBM, KdV, and NLS equations.
           The spatial Fourier-Galerkin semidiscretization uses
           $400$ nodes in $[-50, 50]$ for BBM,
           $1100$ nodes in $[-200, 200]$ for KdV, and
           $1200$ nodes in $[-35, 35]$ for NLS.}
  \label{fig:convergence_in_time}
\end{figure}

\section{Error growth for multiple-soliton solutions}\label{sec:error-growth}

We measure the error growth in time for the two-soliton solutions of the
KdV and NLS equations described in Section~\ref{sec:semidiscrete_conservation}.
Moreover, we consider the three-soliton solution
\cite{hirota1971exact,hietarinta2007introduction}
\begin{equation}
\begin{aligned}
  u &= 12 \, \partial_x^2 \log F,
  \\
  F &= 1 + \e^{\eta_1} + \e^{\eta_2} + \e^{\eta_3}
        + a_{12} \e^{\eta_1 + \eta_2} + a_{13} \e^{\eta_1 + \eta_3}
        + a_{23} \e^{\eta_2 + \eta_3}
        + a_{12} a_{13} a_{23} \e^{\eta_1 + \eta_2 + \eta_3},
  \\
  \eta_i &= k_i (x - x_{i, 0}) - k_i^3 t,
\end{aligned}
\end{equation}
of the KdV equation with parameters
\begin{equation}
  k_1 = 0.75, \quad k_2 = 0.5, \quad k_3 = 0.25,
  \qquad
  x_{1, 0} = -100, \quad x_{2, 0} = 0, \quad x_{3, 0} = 100,
\end{equation}
and the spatial domain $[-400, 400]$ with periodic boundary conditions.
The time span $[0, 1500]$ ensures that the waves interact.
We again use Enzyme.jl \cite{moses2020instead,moses2021reverse} to compute the derivatives.
Since we use a domain twice as large as for the two-soliton solution,
we double the number of nodes accordingly ($N = 2^{10}$ for two solitons,
$N = 2^{11}$ for three solitons) to obtain comparable spatial errors.

For the NLS equation \eqref{eq:nls-u}, we use the three-soliton solution
used in \cite{biswas2024accurate} with $N = 2^{10}$ nodes in the spatial
domain $[-35, 35]$.

\begin{figure}[!htb]
\centering
  \includegraphics[width=\textwidth]{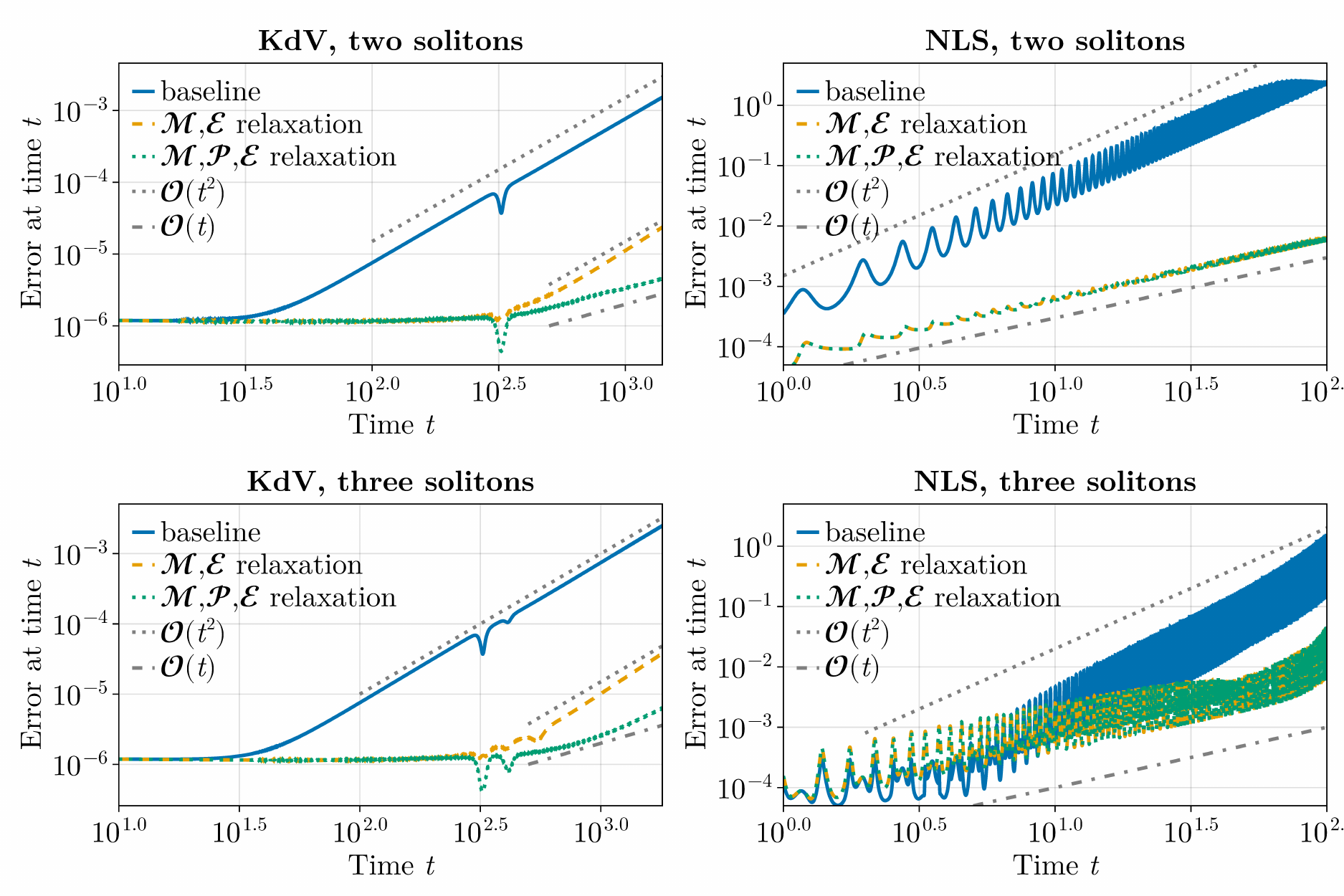}
  \caption{Error growth in time for two- and three-soliton solutions of the
           KdV and NLS equations discretized using Fourier Galerkin methods
           in space.
           The time integration is performed with the fifth-order method of
           \cite{kennedy2019higher} with
           $\Delta t = 0.1$ for the KdV equation as well as
           $\Delta t = 0.01$ (two solitons) and
           $\Delta t = 0.001$ (three solitons) for the NLS equation.}
  \label{fig:error_growth_multiple_solitons}
\end{figure}

The results are shown in Figure~\ref{fig:error_growth_multiple_solitons}.
For the KdV equation, the interaction time of two solitons can be seen
clearly in the error growth plot (the small bumps, e.g., around
$t = 10^{2.5}$). As expected, the error of the baseline method grows
quadratically in time. Before the first soliton interaction, the relaxation
methods  conserving either the mass and energy or all three invariants behave
similarly well. However, after the interaction, the method conserving only
two invariants has a quadratically growing error while the method conserving
all three invariants has a linearly growing error in time. These results
are in accordance with the theoretical predictions for two solitons
\cite{alvarez2010multi}. However, they appear to be better than expected
for three solitons.

\begin{figure}[!htb]
\centering
  \includegraphics[width=\textwidth]{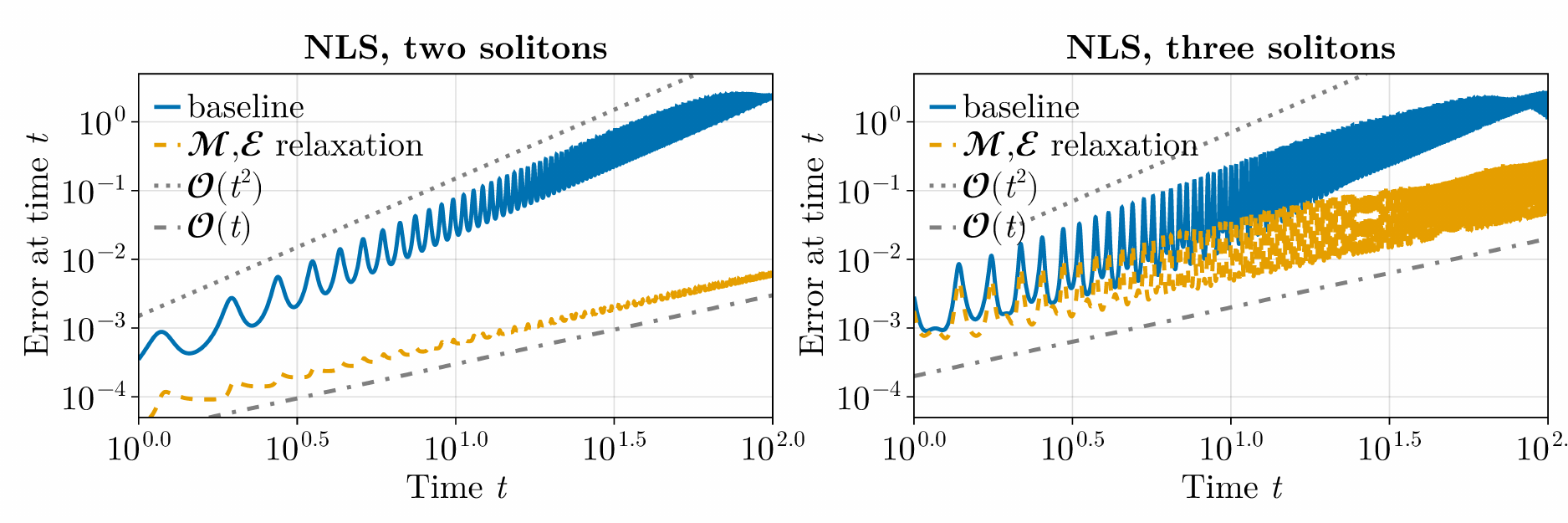}
  \caption{Error growth in time for two- and three-soliton solutions of the
           NLS equation discretized using Fourier collocation methods
           in space.
           The time integration is performed with the fifth-order method of
           \cite{kennedy2019higher} with
           $\Delta t = 0.01$ (two solitons) and
           $\Delta t = 0.002$ (three solitons).}
  \label{fig:error_growth_multiple_solitons_collocation}
\end{figure}

The numerical methods behave differently for the NLS equation. In accordance
with the numerical results of \cite{ranocha2025high} for two solitons,
we observe quadratic error growth in time for the baseline method and a
linear error growth for the mass- and energy-conserving method. However,
conserving the momentum in addition does not improve the results further.
For three solitons, conserving the momentum does also not improve the results
further. Moreover, both relaxation methods result in an eventually
quadratic error growth, in contrast to the numerical results shown in
\cite{ranocha2025high} obtained using a Fourier collocation method in space
that conserves only the mass and energy. The corresponding results are
shown in Figure~\ref{fig:error_growth_multiple_solitons_collocation}.

\subsection{Fourier collocation versus Galerkin methods for NLS}

\begin{figure}[!htb]
\centering
  \includegraphics[width=\textwidth]{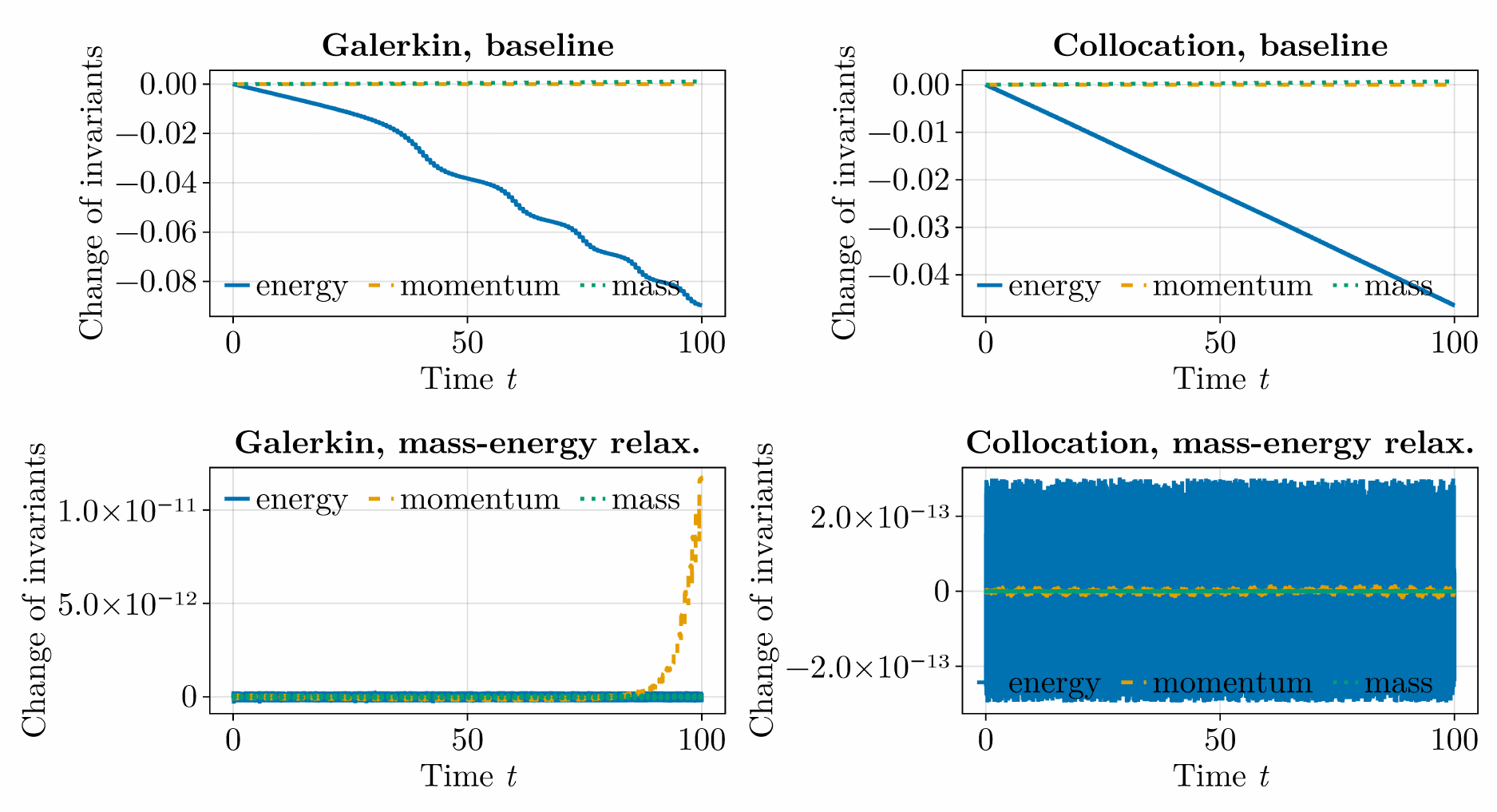}
  \caption{Change of invariants for the three-soliton solution of the
           NLS equation discretized using Fourier Galerkin and collocation
           methods in space.
           The time integration is performed with the fifth-order method of
           \cite{kennedy2019higher} with
           $\Delta t = 0.002$.}
  \label{fig:change_of_invariants_three_solitons_nls}
\end{figure}

To investigate the issue further, we compare the change of the invariants
over time for the three-soliton solutions of the NLS equation using
Fourier Galerkin and collocation methods in space; the results are shown
in Figure~\ref{fig:change_of_invariants_three_solitons_nls}. We clearly
observe that the momentum is much better conserved for the collocation
method with mass-energy relaxation, although there are no theoretical
guarantees for this. This indicates that other properties like symmetry
properties may play a role here and result in the improved error growth
of the collocation method.

Next, we consider the defocusing NLS equation with $\beta = -1$ and
gray/dark soliton solutions. Specifically, we use the one-gray-soliton
solution
\begin{equation}
  u(t, x) = \sqrt{b_0} \e^{\i (\kappa (x - c t) - \omega t)} \left(
    \i \sqrt{\frac{b_1}{b_0}} + \sqrt{1 - \frac{b_1}{b_0}} \tanh\biggl( \sqrt{\frac{b_0 - b_1}{2}} (x - c t)\biggr)
  \right)
\end{equation}
with background mass density $b_0 = 1.5$, minimal mass density $b_1 = 1$,
speed $c = 2 \sqrt{2}$, and derived parameters
\begin{equation}
  \kappa = \frac{c - \sqrt{2 b_1}}{2},
  \quad
  \omega = b_0 - \frac{c^2 - 2 b_1}{4},
\end{equation}
see, e.g., \cite{akhmediev1993first}.
This moving gray soliton decays to a background with constant rate of
change of the phase so that the total momentum is not zero. We computed
the domain boundaries $\pm 31.970600318475647$ such that the gray soliton
is periodic (up to machine precision) and use $2^8$ nodes for the spatial
semidiscretization.

Moreover, we consider the (moving) two-gray-soliton solution
\begin{equation}
  u(t, x) = \e^{\i (k x - k^2 t)} \widetilde{u}(t, x - 2 k t),
  \qquad
  k = 2,
\end{equation}
where
\begin{equation}
  \widetilde{u}(t, x)
  =
  \e^{-\i a_3 t} \frac{(2 a_3 - 4 a_1) \cosh(\mu t / 2) - 2 \sqrt{a_1 a_3} \cosh(2 p x / \sqrt{2}) - \i \mu \sinh(\mu t / 2)}{2 \sqrt{a_3} \cosh(\mu t / 2) + 2 \sqrt{a_1} \cosh(2 p x / \sqrt{2})}
\end{equation}
describes two colliding gray solitons in a steady reference frame
\cite{akhmediev1993first} with background mass density $a_3 = 1.5$,
minimum mass density $a_1 = 1$, and derived parameters
\begin{equation}
  \mu = 4 \sqrt{a_1 (a_3 - a_1)},
  \quad
  p = \sqrt{a_3 - a_1}.
\end{equation}
Since the two solitons collide at $t = 0$, we choose the time span
$[-70, 70]$. The spatial domain boundaries $\pm 409.97784129346803$
are chosen such that the background is periodic (up to machine precision).
We use $2^{11}$ nodes for the spatial semidiscretization.

\begin{figure}[!htb]
\centering
  \includegraphics[width=\textwidth]{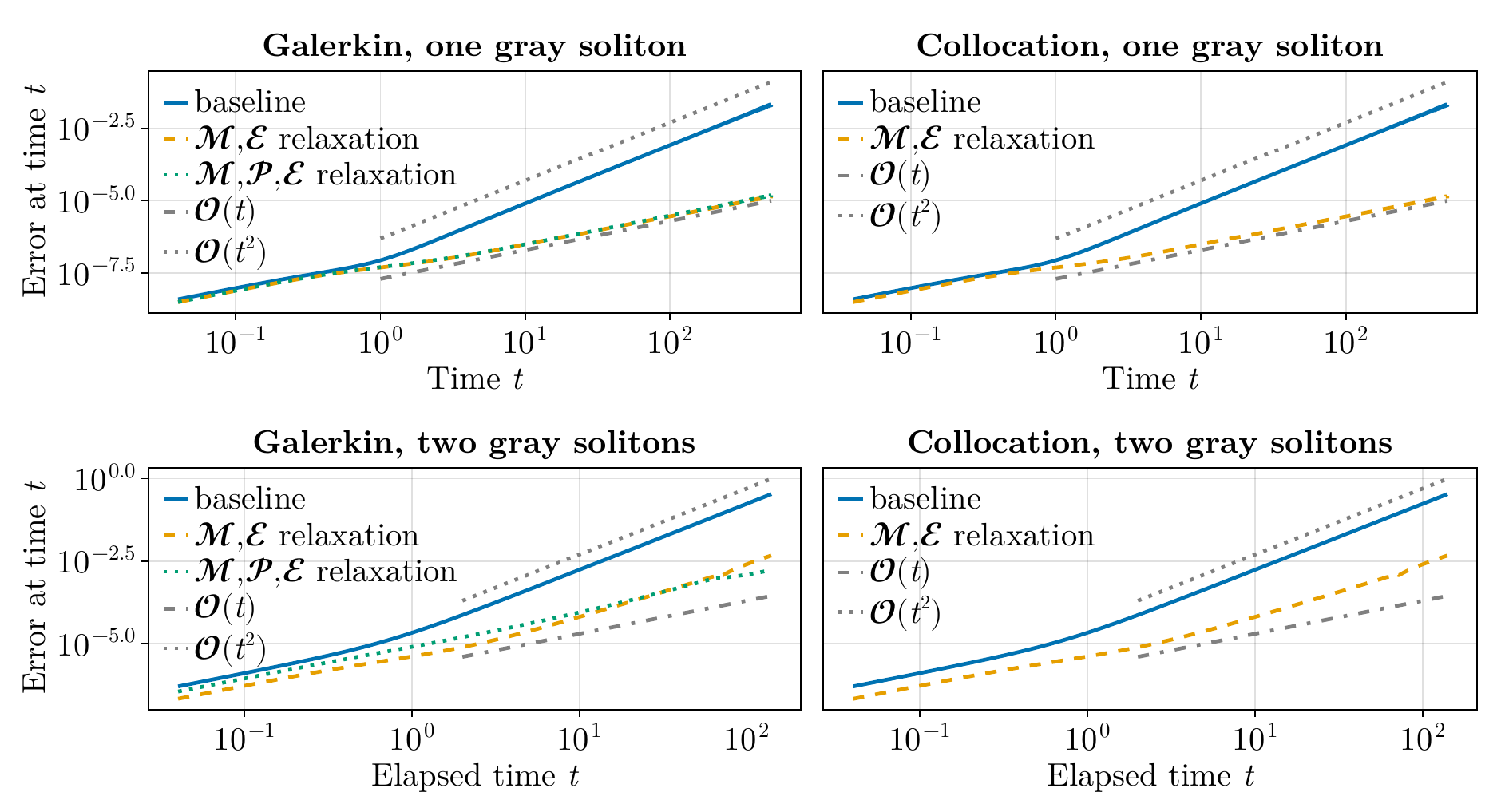}
  \caption{Error growth in time for moving one- and two-gray-soliton
           solutions of the NLS equation.
           The time integration is performed with the fifth-order method of
           \cite{kennedy2019higher} with $\Delta t = 0.04$.
           The two solitons collide after an elapsed time of about
           $t = 70$.}
  \label{fig:error_growth_gray_solitons}
\end{figure}

The error growth in time for these gray soliton solutions is shown in
Figure~\ref{fig:error_growth_gray_solitons}. As expected, the baseline
scheme has a quadratic error growth for all cases. For one gray soliton,
both spatial discretizations (Galerkin and collocation) with relaxation
(only mass and energy for collocation; also momentum for Galerkin) result
in a linear error growth in time. For two gray solitons, we observe
that the mass-, momentum-, and energy-conserving Fourier Galerkin method
with relaxation has a linear error growth while the mass- and
energy-conserving methods have a quadratically growing error. Thus, the
additional conservation of the momentum improves the long-time accuracy,
but increases the error for shorter times a bit.

\begin{figure}[!htb]
\centering
  \includegraphics[width=\textwidth]{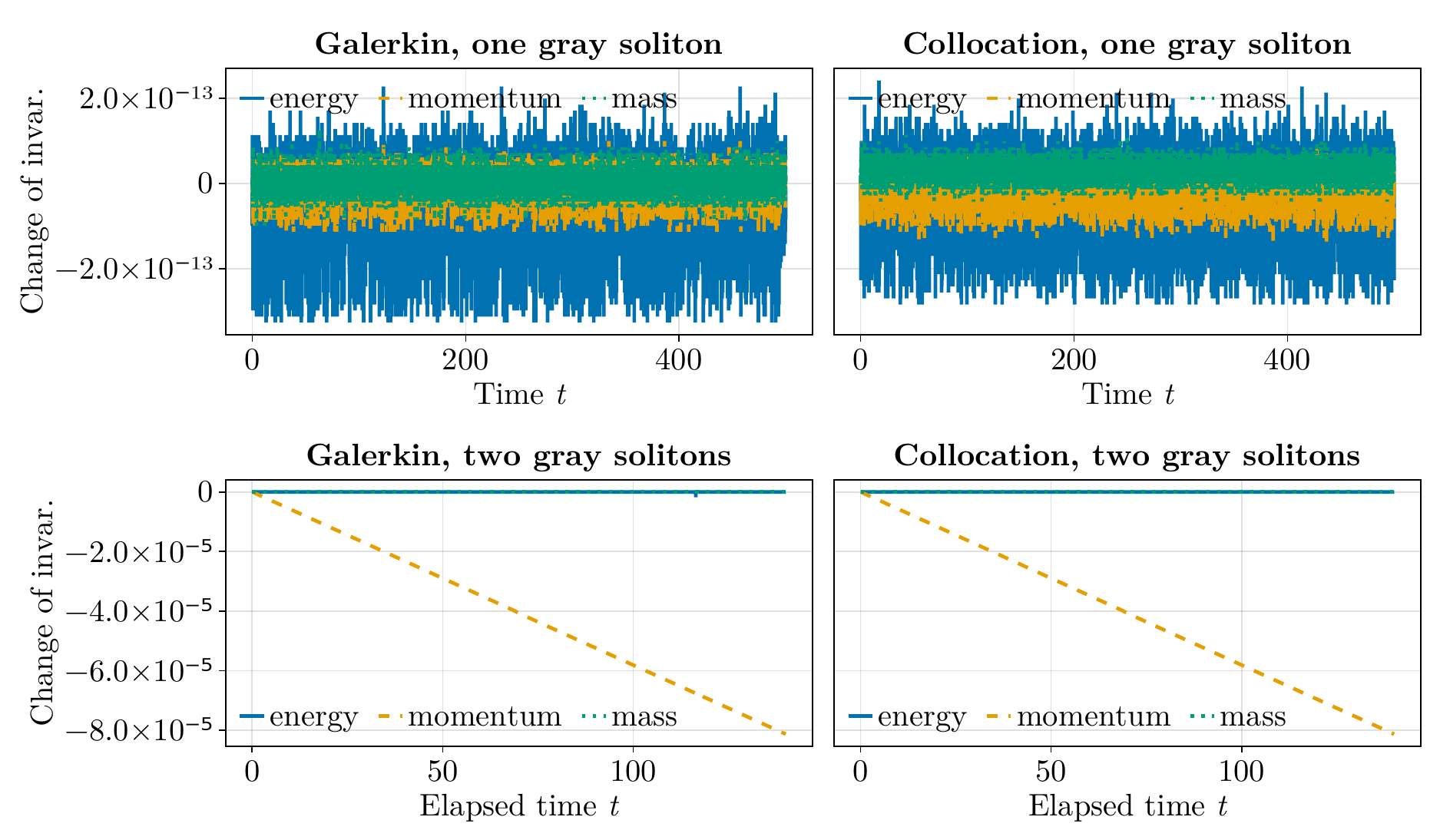}
  \caption{Change of invariants for moving one- and two-gray-soliton
           solutions of the NLS equation with mass- and energy-conserving
           relaxation.
           The time integration is performed with the fifth-order method of
           \cite{kennedy2019higher} with $\Delta t = 0.04$.}
  \label{fig:change_of_invariants_gray_solitons}
\end{figure}

These results differ from the previous findings for the three-soliton
solution of the focusing NLS equation. This is related to the conservation
properties of the relaxation methods enforcing only the mass and energy.
For a moving background, the momentum is essentially conserved for the
one-soliton solution (although this is not guaranteed by relaxation).
However, the momentum varies clearly for the two solitons if relaxation
is used to conserve only the mass and energy, see
Figure~\ref{fig:change_of_invariants_gray_solitons}.

\section{Performance comparisons}
\label{sec:performance_comparisons}

In this section we compare the computational performance of the methods
proposed here with some previous methods from the literature.  First,
we can compare the present method with the scheme we proposed previously for the NLS
equation (that scheme conserves only mass and energy, not momentum).
For both spatial discretizations (Fourier Galerkin here and Fourier
collocation there), the most costly part
is the FFT required to switch between nodal and modal space with
complexity $\mathcal{O}(M \log M)$ for $M$ nodes. Since the NLS has a cubic
nonlinearity, we need $M > 2 N$ for a Galerkin method while we only need $M = N$ for a Fourier
collocation method (conserving mass and energy \cite{ranocha2025high}).
To compute the energy with quartic nonlinearity (e.g., to apply relaxation
or projection to conserve the energy), we even require $M > 3 N / 2$ for the
Galerkin method. Thus, the Fourier Galerkin method is typically (at least)
twice as expensive as the Fourier collocation method for the same number
of degrees of freedom.

\subsection{Comparison with Gauss methods}
\label{sec:comparison_gauss_methods}

Next, we compare our proposed time integration approach to classical high-order, time-reversible/{\allowbreak}symmetric, and symplectic methods.
Thus, we choose the fully-implicit Gauss method \cite[Sections~II.1.3, IV.2.1, and V.2.1]{hairer2006geometric} with $s = 3$ stages, which is a sixth-order method that conserves all linear and quadratic invariants (if the nonlinear systems are solved exactly).
We use the implementation of Gauss methods from GeometricIntegrators.jl \cite{kraus2020GeometricIntegrators}.
In space, we still use the same Fourier Galerkin semidiscretizations as before, both for the Gauss method and the fifth-order method of \cite{kennedy2019higher} with relaxation.
The Gauss method is expected to be more expensive since it needs to solve nonlinear systems with dense Jacobians due to the spectral method and the nonlinear coupling.
Thus, we only present results for a single-soliton solution of the KdV equation as in Section~\ref{sec:fully_discrete_conservation} but with $N = 2^7$ nodes in space.
Since dense linear algebra solvers scale as $\mathcal{O}(N^3)$, we expect the (arguably more interesting) two- and three-soliton simulations to take roughly 500 to 4000 times longer due to the need for increased spatial resolution, making them infeasible to run.
Since the Gauss method has a higher order than the ARK method, we choose a larger time step size $\Delta t = 0.25$ for the Gauss method instead of $\Delta t = 0.1$ for the ARK method to obtain a comparable error in the first steps.
The runtime was measured on an Apple MacBook with M4 CPU.

\begin{figure}[!htb]
\centering
  \includegraphics[width=\textwidth]{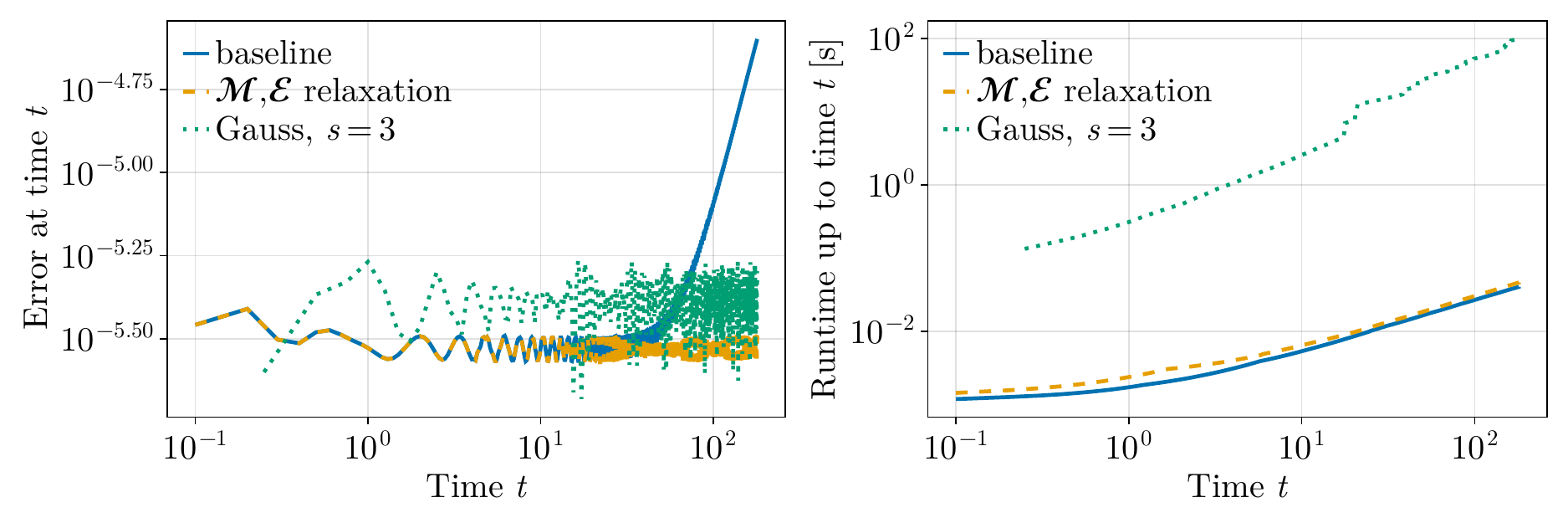}
  \caption{Comparison of the error and runtime of the fifth-order method of \cite{kennedy2019higher} without/with relaxation and the sixth-order Gauss method for a single-soliton solution of the KdV equation.}
  \label{fig:comparison_gauss_methods}
\end{figure}

The results shown in Figure~\ref{fig:comparison_gauss_methods} confirm the expected behavior.
The baseline ARK shows a significantly bigger error growth than both the ARK method with relaxation and the Gauss method.
However, the ARK method with relaxation is roughly three orders of magnitude faster than the Gauss method (100~seconds compared to 0.05~seconds).
This is expected since it avoids the need to solve nonlinear systems and reduces the linear system solves to diagonal systems, thus reducing the complexity from $\mathcal{O}(N^3)$ to $\mathcal{O}(N \log N)$ due to the FFT.

\subsection{Comparison with other methods from the literature}
\label{sec:comparison_other_methods}

For the existing schemes that conserve all three of mass, momentum,
and energy, unfortunately no code is publicly available \cite{akrivis2025high,zheng2024invariants}.
Instead, we compare with two methods that conserve just two quantities and
for which code is available.  All tests were performed on a workstation with two sockets of
20 dual-threaded Intel Xeon Gold 6230 CPUs running Ubuntu 22.04.

The method of Andrews \& Farrell \cite{andrews2025conservative} conserves
mass and energy for BBM.
We apply their implementation and ours to the problem presented in Section
4.2.1 of that work.  Namely, we solve the BBM equation on the domain $x\in[-50,50]$
with periodic boundary conditions and initial condition
\begin{align}
    u(x,t=0) & = \frac{3\sqrt{5}-3}{2} \sech\left(\frac{\sqrt{5}-1}{4} x\right)^2.
\end{align}
We solve up to time $t=2\times10^4$, and use a mesh with 100 degrees of freedom for both
methods.  The method proposed here runs in about 0.40 seconds while that of
Andrews \& Farrell runs in 1141 seconds.

Finally, we also compare with the code of Bai et al.\ \cite{bai2024high}, which
conserves mass and energy for the NLS equation.
We consider the one-soliton problem from Section~4 of that work, given by
\begin{align*}
    u(x,t) & = \operatorname{sech}(x+4t) \exp(-\i (2x+3t)).
\end{align*}
We solve the problem on the domain $x\in[-40,40]$ for $0\le t \le 1$ and measure
the $L^2$ norm of the error at the final time.  Using the
same discretization considered in Figure 4.4(b) of that work the wall clock time
for a run using $\Delta t = 1/512$ and 1024 points in space is approximately 90
seconds and results in an error of $1.26\times 10^{-6}$.
We previously compared our mass- and energy-conserving Fourier collocation
methods with relaxation to their methods in \cite{ranocha2025high}, and found
that our methods were orders of magnitude faster and significantly
more accurate.
Since Fourier Galerkin methods are roughly twice as expensive as the Fourier
collocation methods used in \cite{ranocha2025high}, they are still orders
of magnitude faster than the methods of \cite{bai2024high}.
Specifically, the Fourier collocation code runs in 0.17 seconds, the Fourier
Galerkin code runs in 0.36 seconds, and the code of Bai et al.\ runs in
90 seconds.  The codes proposed here yield errors of $10^{-11}$ or less, while
that of Bai et al.\ yields an error of $1.26\times10^{-6}$.

We also compare the error growth in time of the method of Bai et al.\ \cite{bai2024high} with our scheme.
We use the same one-soliton problem as for the performance comparison reported above, but switch to periodic boundary conditions on the domain $[-40, 40]$ with $N = 1024$ degrees of freedom in space and $\Delta t = 0.01$ for the time integration.

\begin{figure}[!htb]
\centering
  \includegraphics[width=0.5\textwidth]{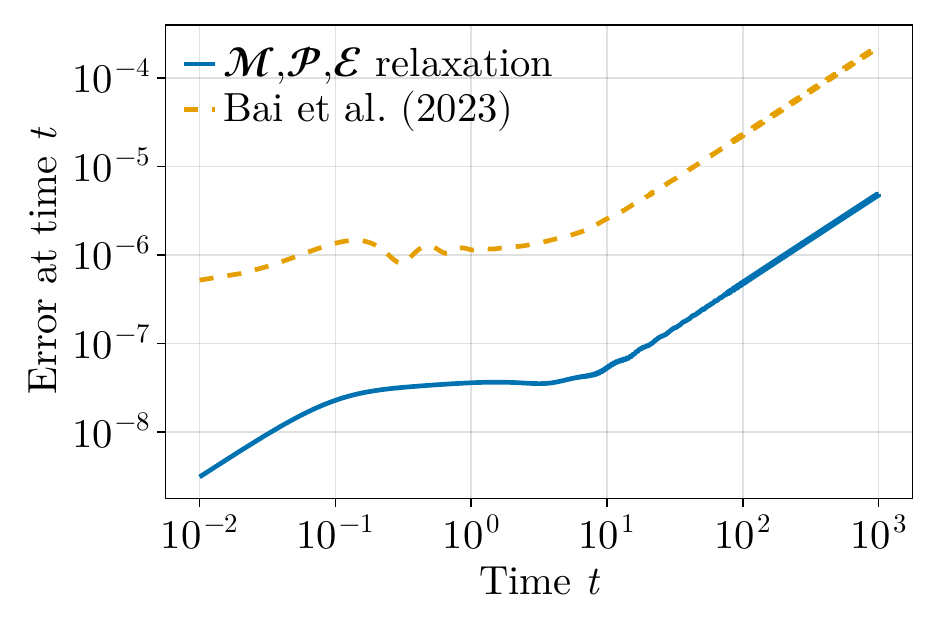}
  \caption{Comparison of the error growth of the Fourier-Galerkin method combined with the fifth-order method of \cite{kennedy2019higher} with relaxation and the method of Bai et al.\ \cite{bai2024high} for a single-soliton solution of the NLS equation.}
  \label{fig:comparison_bai_et_al_error_growth}
\end{figure}

The results shown in Figure~\ref{fig:comparison_bai_et_al_error_growth} are as expected.
Conservation of invariants leads to a linear error growth in time.
Given the same numbers of degrees of freedom and time step size, our method is two orders of magnitude more accurate and maintains the same performance advantage reported earlier.

\section{Hyperbolic approximation of the nonlinear Schr\"odinger equation}
\label{sec:hyperbolization_nls}

We extend the structure-preserving methods to a hyperbolic approximation
of the NLS equation. There are two such hyperbolizations conserving
(appropriate approximations of) the mass, momentum, and energy. The
first one is based on the hydrodynamic formulation and is thus mostly
useful for the defocusing case without vacuum \cite{dhaouadi2019extended};
it conserves the mass and momentum as linear invariants and the energy
as nonlinear invariant. Thus, classical relaxation in time
\cite{ranocha2020relaxation,ranocha2020general}
can be used to conserve all three invariants. Here, we focus on the
second hyperbolization \cite{biswas2025hyperbolic} conserving all three
invariants with a structure similar to the original NLS equation.
This hyperbolization does not require the absence of vacuum. It is given
by \cite{biswas2025hyperbolic}
\begin{equation}
\label{eq:nlsh-q}
\begin{aligned}
  \i q^0_t + q^1_x &= -\beta |q^0|^2 q^0, \\
  \i \tau q^1_t - q^0_x &= -q^1,
\end{aligned}
\end{equation}
where $\tau > 0$ is a relaxation parameter; as $\tau \to 0$, the
solution $q^0$ of the hyperbolization \eqref{eq:nlsh-q} converges formally
to the original NLS equation \eqref{eq:nls-u}. Introducing the real
and imaginary parts $q^0 = v + \i w$ and $q^1 = \nu + \i \omega$ as in
\cite{ranocha2025high}, we obtain
\begin{equation}
\label{eq:nlsh-vw}
\begin{aligned}
  v_t &= -\omega_x - \beta \bigl( v^2 + w^2 \bigr) w, \\
  w_t &= \nu_x + \beta \bigl( v^2 + w^2 \bigr) v, \\
  \tau \nu_t &= w_x - \omega, \\
  \tau \omega_t &= -v_x + \nu.
\end{aligned}
\end{equation}
Using this formulation, the invariants of \eqref{eq:nlsh-q}
\cite{biswas2025hyperbolic} can be written as
\begin{equation}
\label{eq:nlsh-vw_invariants}
\begin{aligned}
  \mass &= \int \left( v^2 + w^2 + \tau \nu^2 + \tau \omega^2 \right) \dif x,
  \\
  \momentum &= \int \left( v w_x - v_x w + \tau \nu \omega_x - \tau \nu_x \omega \right) \dif x = 2 \int \left( v w_x + \tau \nu \omega_x \right) \dif x,
  \\
  \energy &= \int \left( 2 \nu v_x - \nu^2 + 2 \omega w_x - \omega^2 - \frac{\beta}{2} (v^2 + w^2)^2 \right) \dif x.
\end{aligned}
\end{equation}
A Fourier collocation semidiscretization of \eqref{eq:nlsh-vw} conserving
the mass and energy has been developed in \cite{ranocha2025high}. Here, we
consider the Fourier Galerkin semidiscretization
\begin{equation}
\label{eq:nlsh-vw_semi}
\begin{aligned}
  v_t &= -\omega_x - \beta P \bigl( v^2 + w^2 \bigr) w, \\
  w_t &= \nu_x + \beta P \bigl( v^2 + w^2 \bigr) v, \\
  \tau \nu_t &= w_x - \omega, \\
  \tau \omega_t &= -v_x + \nu.
\end{aligned}
\end{equation}
\begin{theorem}
\label{thm:nlsh-vw_semi}
  The Fourier Galerkin semidiscretization \eqref{eq:nlsh-vw_semi}
  of the hyperbolic NLS equation \eqref{eq:nlsh-vw} conserves
  the total mass, momentum, and energy \eqref{eq:nlsh-vw_invariants}.
\end{theorem}
\begin{proof}
  The total mass is conserved since
  \begin{equation}
  \begin{aligned}
    \partial_t \mass
    &=
    2 \langle v, v_t \rangle
    + 2 \langle w, w_t \rangle
    + 2 \langle \nu, \tau \nu_t \rangle
    + 2 \langle \omega, \tau \omega_t \rangle
    \\
    &=
    2 \langle v, -\omega_x - \beta P (v^2 + w^2) w \rangle
    + 2 \langle w, \nu_x + \beta P (v^2 + w^2) v \rangle
    \\
    &\quad
    + 2 \langle \nu, w_x - \omega \rangle
    + 2 \langle \omega, -v_x + \nu \rangle
    \\
    &=
    -2 \beta \int (v^2 + w^2) (v w - w v) \dif x
    =
    0.
  \end{aligned}
  \end{equation}
  Similarly, the total momentum is conserved since
  \begin{equation}
  \begin{aligned}
    \partial_t \momentum
    &=
    2 \langle w_x, v_t \rangle
    - 2 \langle v_x, w_t \rangle
    + 2 \langle \omega_x, \tau \nu_t \rangle
    - 2 \langle \nu_x, \tau \omega_t \rangle
    \\
    &=
    2 \langle w_x, -\omega_x - \beta P (v^2 + w^2) w \rangle
    - 2 \langle v_x, \nu_x + \beta P (v^2 + w^2) v \rangle
    \\
    &\quad
    + 2 \langle \omega_x, w_x - \omega \rangle
    - 2 \langle \nu_x, -v_x + \nu \rangle
    \\
    &=
    -2 \beta \int (v^2 + w^2) (w w_x + v v_x) \dif x
    =
    - \beta \int (v^2 + w^2)^2_x \dif x
    =
    0.
  \end{aligned}
  \end{equation}
  Finally, the total energy is conserved since
  \begin{equation}
  \begin{aligned}
    \partial_t \energy
    &=
    - 2 \langle \nu_x, v_t \rangle
    - 2 \beta \langle (v^2 + w^2) v, v_t \rangle
    - 2 \langle \omega_x, w_t \rangle
    - 2 \beta \langle (v^2 + w^2) w, w_t \rangle
    \\
    &\quad
    + 2 \langle v_x, \nu_t \rangle
    - 2 \langle \nu, \nu_t \rangle
    + 2 \langle w_x, \omega_t \rangle
    - 2 \langle \omega, \omega_t \rangle
    \\
    &=
    - 2 \langle \nu_x, -\omega_x - \beta P (v^2 + w^2) w \rangle
    - 2 \beta \langle (v^2 + w^2) v, -\omega_x - \beta P (v^2 + w^2) w \rangle
    \\
    &\quad
    - 2 \langle \omega_x, \nu_x + \beta P (v^2 + w^2) v \rangle
    - 2 \beta \langle (v^2 + w^2) w, \nu_x + \beta P (v^2 + w^2) v \rangle
    \\
    &\quad
    + 2 \tau^{-1} \langle v_x, w_x - \omega \rangle
    - 2 \tau^{-1} \langle \nu, w_x - \omega \rangle
    \\
    &\quad
    + 2 \tau^{-1} \langle w_x, -v_x + \nu \rangle
    - 2 \tau^{-1} \langle \omega, -v_x + \nu \rangle
    =
    0.
  \end{aligned}
  \end{equation}
  As before, we used the exactness of the $L^2$ projection $P$ for these
  computations.
\end{proof}

To apply mass- and energy-conserving relaxation methods, we use the
simplified projection method to conserve the total mass described in
\cite{ranocha2025high}, i.e., we scale $v$, $w$ by $\alpha_1$
and $\nu$, $\omega$ by $\alpha_2$, where
\begin{equation}
\label{eq:mass_projection_nlsh}
\begin{aligned}
  \alpha_1
  &=
  \frac{p^2 (\tau - 1) \tau^2 + \sqrt{-p^2 q^2 (\tau - 1)^2 \tau + c (q^2 + p^2 \tau^3)}}{q^2 + p^2 \tau^3},
  \\
  \alpha_2
  &=
  \frac{q^2 (1 - \tau) + \tau \sqrt{-p^2 q^2 (\tau - 1)^2 \tau + c (q^2 + p^2 \tau^3)}}{q^2 + p^2 \tau^3},
\end{aligned}
\end{equation}
$c$ is the desired value of the mass \eqref{eq:nlsh-vw_invariants}, and
\begin{equation}
  q^2 = \| v \|_{L^2}^2 + \| w \|_{L^2}^2,
  \qquad
  p^2 = \| \nu \|_{L^2}^2 + \| \omega \|_{L^2}^2.
\end{equation}

\begin{figure}[!htb]
\centering
  \includegraphics[width=\textwidth]{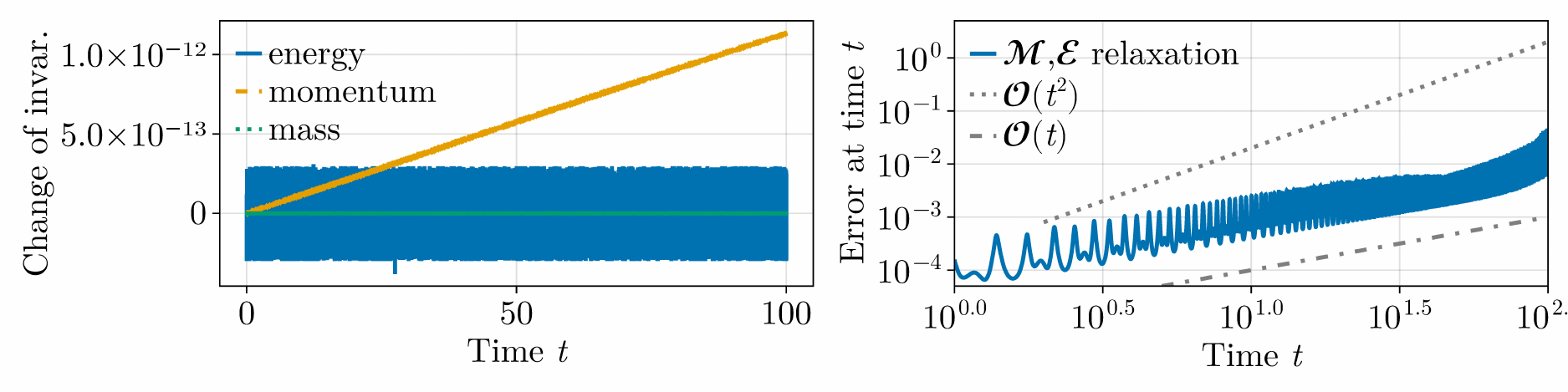}
  \caption{Numerical results for the hyperbolized NLS equation with
           $\tau = 10^{-9}$ initialized with the three-soliton solution
           of the NLS equation, which is used to compute the error.
           The time integration is performed with the fifth-order method of
           \cite{kennedy2019higher} with $\Delta t = 0.001$ and
           mass- and energy-conserving relaxation.}
  \label{fig:hyperbolized_nls}
\end{figure}

The hyperbolization behaves similarly to the original NLS equation.
Thus, we only show a subset of numerical results.
Figure~\ref{fig:hyperbolized_nls} demonstrates that the results agree
very well with the original NLS equation for $\tau = 10^{-9}$.
In particular, we observe an error growth behavior similar to
Figure~\ref{fig:error_growth_multiple_solitons}; mass and energy are
conserved as expected, and the momentum grows slightly (since it is not
enforced to be constant by relaxation here).

\section{Conclusions and future directions}
\label{sec:summary}
The schemes we have proposed are capable of preserving mass, momentum, and
energy to machine precision, while also running faster (by orders of magnitude)
than other recently-proposed schemes that conserve only two of these
quantities.  Furthermore, they can be of arbitrarily high order, simply by using
sufficiently high-order Runge-Kutta methods in time.  Their performance advantage
over implicit methods is due to the fact that they only require the solution
of a scalar algebraic equation at each step.

For the numerical solution of integrable systems, it is natural to ask whether
similar schemes could be designed to conserve even more invariants.  This seems
difficult in terms of both the spatial and temporal discretization.  We have found
that Fourier Galerkin methods do not preserve, for instance, the fourth invariant
of KdV.  Meanwhile, the efficiency of our projection-based relaxation method
relies crucially on the efficiency of the inner projection onto the manifolds
defined by the first invariants. While the idea can be extended directly to
projections that conserve $N$ invariants (with relaxation to conserve invariant
$N+1$), efficiency will likely be significantly reduced if the inner projections
are too expensive.  The Toda lattice and the Ablowitz-Ladik lattice are
natural systems on which to test time discretizations that may preserve several
invariants.

While our description and experiments have been restricted to one spatial
dimension, the schemes can be extended in a natural way to higher dimensions.
In order to conserve momentum, we require the imposition of periodic boundary
conditions.  With more general boundary conditions, mass and energy are still
conserved, but we note that conservation of mass and energy can be achieved
also with other SBP spatial discretizations along with previously-published
time discretization approaches.
Another natural question is whether the effort to conserve three invariants is
worthwhile.  Of course, conservation is already
highly desirable from a physical point of view, and our experiments show that
in some cases conserving all three quantities gives a significant quantitative
advantage over conserving just two, with respect to long-time error growth.
In the end, the choice of whether to use a structure-preserving method depends
on the application and the goals of the computation.
The downside of imposing all three discrete conservation laws is that one must
use Fourier Galerkin methods, which cost notably more than collocation methods.
A more detailed comparison of the cost of different conservation techniques
(especially with respect to time) is the subject of future work.

%% file: funding.tex
HR was supported by the Deutsche Forschungsgemeinschaft
(DFG, German Research Foundation, project numbers 513301895 and 528753982
as well as within the DFG priority program SPP~2410 with project number 526031774).
DK was supported by funding from King Abdullah University of Science and Technology.